\documentclass[a4paper,11pt]{article}

\usepackage[left=3cm, right=3cm]{geometry}
\usepackage{amsfonts,amssymb,amsmath,amsthm} 
\usepackage[utf8]{inputenc} 
\usepackage[alphabetic, nobysame]{amsrefs}
\usepackage{enumerate}
\usepackage[english]{babel}
\usepackage{authblk}
\usepackage[colorlinks, pagebackref=false]{hyperref}
\usepackage{pstricks-add}

\newtheorem{theorem}{Theorem}[section]
\newtheorem{lemma}[theorem]{Lemma}
\newtheorem{proposition}[theorem]{Proposition}
\newtheorem{corollary}[theorem]{Corollary}
\newtheorem{claim}{Claim}
\newtheorem{fact}{Fact}

\newenvironment{factproof}[1]{\par\noindent\textit{Proof of the fact:}\space#1}{\hfill $\blacksquare$}
\newenvironment{claimproof}[1]{\par\noindent\textit{Proof of the claim:}\space#1}{\hfill $\blacksquare$}

\theoremstyle{definition}

\newtheorem{remark}[theorem]{Remark}

\newcommand{\suchthat}{\;\ifnum\currentgrouptype=16 \middle\fi|\;}

\newcommand{\leftexp}[2]{{\vphantom{#1}}^{#2}{#1}}
\newcommand{\closure}[2]{\leftexp{\overline{#2}}{#1}}

\newcommand{\bigslant}[2]{{\raisebox{.2em}{$#1$}\left/\raisebox{-.2em}{$#2$}\right.}}
\DeclareMathOperator{\Aut}{\mathrm{Aut}}
\DeclareMathOperator{\Z}{\mathbf{Z}}
\DeclareMathOperator{\R}{\mathbf{R}}
\DeclareMathOperator{\N}{\mathbf{Z}_{\geq 0}}

\DeclareMathOperator{\id}{\mathrm{id}}
\DeclareMathOperator{\Sym}{\mathrm{Sym}}
\DeclareMathOperator{\Alt}{\mathrm{Alt}}
\DeclareMathOperator{\Ch}{\mathrm{Ch}}
\DeclareMathOperator{\Mon}{\mathrm{Mon}}
\DeclareMathOperator{\Sub}{\mathbf{Sub}}
\DeclareMathOperator{\bd}{\partial\!}

\title{Chabauty limits of simple groups acting on trees}
\author[1]{Pierre-Emmanuel Caprace\thanks{F.R.S.-FNRS Senior Research Associate, supported in part by the ERC (grant \#278469).}}
\author[1]{Nicolas Radu\thanks{F.R.S.-FNRS Research Fellow.}}

\affil[1]{UCLouvain, 1348 Louvain-la-Neuve, Belgium}
\date{July 4, 2018}

\begin{document}

\maketitle

\begin{abstract}
Let $T$ be a locally finite tree without vertices of degree~$1$. We show that among the closed subgroups of $\Aut(T)$ acting with a bounded number of orbits, the Chabauty-closure of the set of topologically simple groups is the set of groups without proper open subgroup of finite index. Moreover, if all vertices of $T$ have degree $\geq 3$, then the set of isomorphism classes of topologically simple closed subgroups of $\Aut(T)$ acting doubly transitively on $\bd T$ carries a natural compact Hausdorff topology inherited from Chabauty. Some of our considerations are valid in the context of automorphism groups of locally finite connected graphs. Applications to Weyl-transitive automorphism groups of buildings are also presented. 
\end{abstract}

{\footnotesize
\tableofcontents
}

\section{Introduction}
\begin{flushright}
\begin{minipage}[t]{0.35\linewidth}\itshape\small
H\^etre c'est mon identit\'e

\hfill\upshape (Jacques Pr\'evert, \emph{Arbres}, 1976)
\end{minipage}
\end{flushright}

Beyond algebraic groups over local fields, groups acting on trees provide the largest (and historically the first) known source of examples of non-discrete compactly generated locally compact groups that are \textbf{topologically simple}, i.e.\ whose only closed normal subgroups are the trivial ones. Since the automorphism group of a given locally finite tree $T$ may host many pairwise non-isomorphic topologically simple closed subgroups, it is natural to consider those collectively, by viewing them as a subset of the space $\Sub(\Aut(T))$ of all closed subgroups of $\Aut(T)$, endowed with the Chabauty topology, which is compact. The starting point of this work is the following basic question: \textit{what is the Chabauty-closure of the set of topologically simple closed subgroups of $\Aut(T)$?} In order to stay in the realm of compactly generated groups, we will frequently impose that the groups under consideration act with a bounded number of orbits. Assuming the weaker condition that the groups act cocompactly on $T$ is sufficient to guarantee that they are compactly generated, but that condition is not Chabauty-closed. To facilitate the statements of our results, we introduce the following notation. For a given number $C>0$, we denote by 
$$\Sub(\Aut(T))_{\leq C}$$
the set of closed subgroups of $\Aut(T)$ acting with at most $C$ orbits of vertices. It is a clopen subset of $\Sub(\Aut(T))$ (see Proposition~\ref{proposition:CLOPEN} (3)).

\begin{theorem}\label{theorem:LimitsOfSimple}
Let $T$ be a locally finite tree all of whose vertices have degree~$\geq 2$. For any $C > 0$, the Chabauty-closure of the set of topologically simple groups in $\Sub(\Aut(T))_{\leq C}$ is the set of groups in $\Sub(\Aut(T))_{\leq C}$ without proper open subgroup of finite index.
\end{theorem}

The conclusion of Theorem~\ref{theorem:LimitsOfSimple} may fail if the tree $T$ is allowed to have vertices of degree~$1$, see Lemma~\ref{lemma:Valency1} below. 

\medskip

Following Burger--Mozes \cite{BM}, it is customary to denote the intersection of all open subgroups of finite index in a given locally compact group $H$ by $H^{(\infty)}$. We also denote by $\Mon(H)$ the \textbf{monolith} of $H$, i.e.\ the (possibly trivial) intersection of all non-trivial closed normal subgroups of $H$. Notice that $H$ is topologically simple if and only if $H = \Mon(H)$. With these notations at hand, the statement of Theorem~\ref{theorem:LimitsOfSimple} can be epitomized by the following equality:
$$\overline{\{H \in \Sub(\Aut(T))_{\leq C} \mid H = \Mon(H)\} } = \{H \in \Sub(\Aut(T))_{\leq C} \mid H = H^{(\infty)}\}.$$

We remark that if $C = 1$ then the set $\{H \in \Sub(\Aut(T))_{\leq C} \mid H = H^{(\infty)}\}$ is empty, while if $C \geq 2$ and $T$ is \textbf{semi-regular} (i.e.\ $\Aut(T)$ is edge-transitive), that set contains at least one group, namely the group $\Aut(T)^+$ of type-preserving automorphisms, which is simple by \cite{Tits_arbre}. For a general tree $T$ and an arbitrarily large $C$, it may be the case that $\Sub(\Aut(T))_{\leq C}$ contains only discrete, hence virtually free, groups (see \cite{BassTits}), so that the set $\{H \in \Sub(\Aut(T))_{\leq C} \mid H = H^{(\infty)}\}$ is also empty in that case.

It is important to note that a Chabauty limit of topologically simple groups need not be simple. Indeed, explicit examples of non-simple closed subgroups $H$ of $\Aut(T)$ that are edge-transitive (indeed locally $2$-transitive) and satisfy $H = H^{(\infty)}$ are provided by Burger and Mozes in \cite{BM}*{Example 1.2.1} (see also Remark~\ref{rem:BuMo} below). Thus the set of topologically simple edge-transitive closed subgroups  is not closed in $\Sub(\Aut(T))$. Nevertheless, that situation changes if one considers the subset of groups acting doubly transitively on the set of ends of a \textit{thick} tree $T$ (which is automatically contained in $\Sub(\Aut(T))_{\leq 2}$, see \cite{BM}*{Lemma~3.1.1}). Recall that $T$ is \textbf{thick} if all its vertices have degree $\geq 3$, and remark that $\Sub(\Aut(T))_{\leq 2}$ is non-empty only when $T$ is semi-regular.

\begin{theorem}\label{theorem:Bd-2-trans}
Let $T$ be a locally finite thick semi-regular tree. The set of topologically simple closed subgroups of $\Aut(T)$ acting $2$-transitively on $\bd T$ is Chabauty-closed. 

Moreover, the isomorphism relation within that set has closed classes, and  the set $\mathcal{S}_T$ of isomorphism classes of topologically simple groups acting continuously and properly on $T$ and $2$-transitively on $\bd T$, endowed with the quotient topology, is compact Hausdorff.
\end{theorem}

Theorem~\ref{theorem:Bd-2-trans} has several consequences. First of all, it can be interpreted as providing qualitative information on the complexity of the isomorphism relation within topologically simple boundary-$2$-transitive closed subgroups of $\Aut(T)$. Indeed, Theorem~\ref{theorem:Bd-2-trans} implies that that relation is \emph{smooth} in the sense of \cite{Gao}*{Definition~5.4.1}, which means that it comes at the bottom of the hierarchy of complexity of classification problems in the formalism established by invariant descriptive set theory (see \cite{Gao}*{Chapter~15}). In fact, it is tantalizing to believe that for a given tree $T$, the set $\mathcal{S}_T$ of isomorphism classes as above can be described exhaustively. This has actually recently been accomplished by the second-named author for all semi-regular trees whose vertex degrees are $\geq 6$ and such that the only finite $2$-transitive groups of those degrees are the full symmetric or alternating groups, see \cite{Radu} and Appendix~\ref{sec:appendix} below. For all those trees, the set $\mathcal{S}_T$ happens to be countable. Moreover, the second Cantor--Bendixson derivative of $\mathcal{S}_T$ is reduced to the singleton consisting of the isomorphism class of the group $\Aut(T)^+$ (see Proposition~\ref{proposition:Cantor-Bendixson} and Remark~\ref{remark:S^Alt} below). However, the classification problem remains open for semi-regular trees $T$ whose vertex degrees are the degrees of smaller finite $2$-transitive groups, like Lie-type groups or affine groups. In particular, we do not know whether there exists a tree $T$ such that $\mathcal{S}_T$ is uncountable. The case of the trivalent tree is especially intriguing. 

The compactness of $\mathcal{S}_T$ asserted by Theorem~\ref{theorem:Bd-2-trans} also fosters less ambitious hope than a full classification of $\mathcal S_T$. Indeed, it opens up the possibility to find new isomorphism types of simple groups by taking limits of known ones. Implementing this idea requires to have at hand an infinite family of pairwise non-isomorphic topologically simple groups acting boundary-$2$-transitively on the same locally finite tree $T$. Rank one simple algebraic groups over $p$-adic fields provide examples of such families. However, in all cases where it could be verified, any limit of (classes of) such groups in $\mathcal S_T$ happens to be a rank one simple algebraic groups over a local field of positive characteristic. Indeed, T.\ Stulemeijer has proved that if $T$ is the regular tree of degree $p+1$ with $p$ prime, then the set of isomorphism classes of algebraic groups in $\mathcal{S}_T$, denoted by $\mathcal{S}_T^{\mathrm{alg}}$, is closed. Moreover the non-isolated points are precisely the isomorphism classes of the simple algebraic groups over local fields of positive characteristic. That set is finite (of cardinality~$2$) if $p > 2$ and infinite if $p=2$. We refer to \cite{Stulemeijer} for general results and full details. 

Another potential source of examples for the implementation of that idea is the class of complete Kac--Moody groups of rank two over finite fields. In that class, the tree $T$ is determined by the finite ground field. Letting the defining generalized Cartan matrix run over the infinite set of possibilities in rank two, one obtains a countable family of topologically simple boundary-$2$-transitive groups in $\Sub(\Aut(T))$. The difficulty arising here is that we do not know whether those groups are pairwise non-isomorphic: we do not even know whether they form infinitely many isomorphism classes. A discussion of this rather subtle question, and partial answers, may be found in \cite{Marquis}*{Theorem~F and \S6}.

\medskip

An important tool in the proofs of the results above is provided by the notion of \textit{$k$-closures} recently introduced by Banks--Elder--Willis \cite{Banks}, some of whose properties are reviewed in \S\ref{section:BEW} below. We establish a key relation between Chabauty convergence and  $k$-closures in the general context of automorphism groups of locally finite graphs, see Proposition~\ref{proposition:convergence_k}. We deduce the following statement, which is the main intermediate step in the proof of Theorem~\ref{theorem:LimitsOfSimple}. 

\begin{theorem}\label{theorem:simpleclosed}
Let $\Lambda$ be a locally finite connected (simple, undirected) graph and $\Gamma \leq \Aut(\Lambda)$ act cocompactly on $\Lambda$. Let $H_n \to H$ be a converging sequence in $\Sub({\Aut(\Lambda)})$. Suppose that for each $n \geq 1$, there exists $\tau_n \in \Aut(\Lambda)$ such that $\tau_n \Gamma \tau_n^{-1} \leq H_n$. Then we have
$$[H : H^{(\infty)}] \leq \limsup_{n\to \infty}\, [H_n : H_n^{(\infty)}].$$
In particular, the set 
$$\{H \in \Sub(\Aut(\Lambda)) \mid H \geq \Gamma \text{ and } H = H^{(\infty)}\}$$ 
is Chabauty-closed.
\end{theorem}

The condition that all groups $H_n$ contain a conjugate of a fixed group $\Gamma$ acting cocompactly may be viewed as a strengthening of the condition bounding the number of orbits, which was imposed in Theorem~\ref{theorem:LimitsOfSimple}. Classical results by Bass \cite{Bass} and Bass--Kulkarni \cite{Bass-Kulkarni} ensure that when $\Lambda$ is a tree, both conditions are equivalent (see \S\ref{subsection:trees} below). Building upon this, we tighten the relation between Chabauty convergence of unimodular cocompact subgroups of $\Aut(\Lambda)$ and $k$-closures (see Corollary~\ref{corollary:tree:convergence<k-closure}) and deduce that the algebraic properties of   \textit{local pro-$\pi$-ness}  and \textit{local torsion-freeness} are both Chabauty-open in that context, see Propositions~\ref{proposition:LocalPrimeContent} and~\ref{proposition:LocallyTorsionFree}.

\medskip

Taking advantage of the rather flexible hypotheses of Theorem~\ref{theorem:simpleclosed}, we include applications to groups acting on buildings that are not necessarily trees, see Corollary~\ref{corollary:Building}. We are not aware of families of graphs other than trees where analogues of the aforementioned results by Bass--Kulkarni hold. However, we note that chamber-transitive buildings whose Weyl group is virtually free all admit a canonical continuous proper cocompact action on a tree (see Lemma~\ref{lemma:virtuallyfree}), so that the condition that the groups under consideration all contain a conjugate of a fixed group $\Gamma$ also becomes redundant in that context, see Corollary~\ref{corollary:virtuallyfree}. 

\subsection*{Acknowledgement}

We thank the anonymous referee for constructive comments and suggestions.

\section{The Chabauty space}\label{sec:Chabauty}

Given a locally compact group $G$, we denote by $\Sub(G)$ the set of closed subgroups of $G$ equipped with the Chabauty topology, which is compact Hausdorff (see \cite{Bourbaki}*{Chapitre VIII, \S5, no. 3, Th\'eor\`eme 1}). Recall that a base of neighborhoods of $H \in \Sub(G)$ in the \textbf{Chabauty topology} is given by the sets
$$\mathcal{V}_{K,U}(H) := \{J \in \Sub(G) \mid J \cap K \subseteq HU \text{ and } H \cap K \subseteq JU\},$$
where $K$ ranges over compact subsets of $G$ and $U$ over non-empty open subsets of $G$. 

Assume that $G$ is second countable. In that case,  the compact space $\Sub(G)$ is also second countable. In particular $\Sub(G)$ is metrizable by Urysohn's Metrization Theorem (alternatively, one may directly define a compatible metric on $\Sub(G)$, see \cite{Gelander}*{Exercise~2}). The locally compact groups appearing in this paper will mostly be automorphism groups of connected locally finite graphs: endowed with the compact open topology, those are second countable (totally disconnected) locally compact groups.

\begin{lemma}\label{lemma:converge}
Let $G$ be a second countable locally compact  group. A sequence $(H_n)$ in $\Sub(G)$ converges to $H \in \Sub(G)$ if and only if the two conditions below are satisfied:
\begin{enumerate}[(i)]
\item Let $(H_{k(n)})$ be a subsequence of $(H_n)$ and let $(h_{k(n)})$ be a sequence in $G$ such that $h_{k(n)} \in H_{k(n)}$ for each $n \geq 1$. If $(h_{k(n)})$ converges to $h \in G$, then $h \in H$.
\item Any $h \in H$ is the limit of a sequence $(h_n)$ with $h_n \in H_n$ for each $n \geq 1$.
\end{enumerate}
\end{lemma}

\begin{proof}
See~\cite{Guivarch}*{Lemma~2}.
\end{proof}

The following results are then immediate.

\begin{lemma}\label{lemma:conjugate}
Let $G$ be a second countable locally compact  group. The conjugation action of $G$ on $\Sub(G)$ is jointly continuous, i.e.\ if $g_n \to g$ is a converging sequence in $G$ and $H_n \to H$ is a converging sequence in $\Sub(G)$, then $g_n H_n g_n^{-1} \to g H g^{-1}$.
\end{lemma}

\begin{proof}
This is an easy consequence of Lemma~\ref{lemma:converge}.
\end{proof}

\begin{lemma}\label{lemma:inter}
Let $G$ be a second countable locally compact  group.
\begin{enumerate}[(1)]
\item If $(H_n)$ is a descending chain in $\Sub(G)$, then $H_n \to \bigcap_{i \geq 1} H_i$;
\item If $(H_n)$ is an ascending chain in $\Sub(G)$, then $H_n \to \overline{\bigcup_{i \geq 1} H_i}$.
\end{enumerate}
\end{lemma}

\begin{proof}
We prove (1), the proof of (2) being similar. Let us check (i) and (ii) in Lemma~\ref{lemma:converge}. Any $h \in \bigcap_{i \geq 1} H_i$ is the limit of the constant sequence $(h)$, so (ii) is clear. Now in order to prove (i), let $h_{k(n)} \to h$ be a converging sequence in $G$ such that $h_{k(n)} \in H_{k(n)}$ for each $n \geq 1$. For each $i \geq 1$, the sequence $(h_{k(n)})_{k(n)\geq i}$ is contained in $H_i$. Since $H_i$ is closed and $h_{k(n)} \to h$, we get $h \in H_i$. This being true for any $i \geq 1$, we have $h \in \bigcap_{i \geq 1} H_i$.
\end{proof}

We also record the following essential result for the sake of future references. 

\begin{theorem}\label{theorem:Unimod}
Let $G$ be a locally compact group. The set $\Sub(G)^0$ of unimodular closed subgroups of $G$ is closed in $\Sub(G)$.
\end{theorem}
\begin{proof}
See \cite{Bourbaki}*{Chapitre VIII, \S5, no. 3, Th\'eor\`eme 1}.
\end{proof}

The next basic lemma plays a key role in the proof of Proposition~\ref{proposition:CLOPEN}.

\begin{lemma}\label{lemma:basicChabauty}
Let $G$ be a locally compact group and $C$ be a compact open subset of $G$ (e.g. $C$ is a coset of a compact open subgroup). Then the set
$$\{H \in \Sub(G) \mid H \cap C \neq \varnothing\}$$
is clopen in $\Sub(G)$.
\end{lemma}

\begin{proof}
Since $C$ is compact, we have $\bigcap_U CU^{-1} = C$, where the intersection is taken over all open, relatively compact, identity neighbourhoods $U$ in $G$. Since $C$ is also open, it follows that   there exists an open, relatively compact, identity neighborhood $U$ in $G$ such that $CU^{-1} = C$. For any $H \in \Sub(G)$, we then consider the basic Chabauty-neighborhood
$$\mathcal{V}_{C,U}(H) = \{J \in \Sub(G) \mid J \cap C \subseteq HU \text{ and } H \cap C \subseteq JU\}$$
of $H$. We observe that, for any $J \in \mathcal{V}_{C,U}(H)$, we have $J \cap C \neq \varnothing$ if and only if $H \cap C \neq \varnothing$. Thus  the set
$\{H \in \Sub(G) \mid H \cap C = \varnothing\}$ and its complement $\{H \in \Sub(G) \mid H \cap C \neq \varnothing\}$ are both open. 
\end{proof}

In the following proposition, as well as in the rest of the paper, we adopt the terminology from \cite{Bass}*{\S 1} concerning graphs. Given a graph $\Lambda$ and a group $H \leq \Aut(\Lambda)$ acting without inversion on $\Lambda$, one can form the quotient graph $H \backslash \Lambda$ and the canonical projection $p \colon \Lambda \to H \backslash \Lambda$. We recall from  \cite{BassLubotzky}*{\S 2.5} that the quotient graph $Q = H \backslash \Lambda$ is an \textbf{edge-indexed graph}, i.e.\ it comes equipped with the map $i$ associating to each oriented edge $e$ of $Q$ (where a geometric edge is seen as a pair of oriented edges) the cardinal
$$i(e) = \# \{a \in E(\Lambda) \mid p(a) = e \text{ and $x$ is the origin of $a$}\},$$
where $x$ is any vertex in $\Lambda$ such that $p(x)$ is the origin vertex of $e$. 
Since $H$ always acts without inversion on the first barycentric subdivision $\Lambda^{(1)}$ of $\Lambda$, it follows that any group $H \leq \Aut(\Lambda)$ yields a well defined edge-indexed quotient graph $H \backslash \Lambda^{(1)}$.

We also need to define a \textbf{coloring} of a graph $\Lambda$ as a map $c \colon V(\Lambda) \to \mathcal{C}$, where $\mathcal{C}$ is any set. We write $\Lambda_c$ for $\Lambda$ considered with its coloring $c$ and $\Aut(\Lambda_c)$ for the group of all automorphisms of $\Lambda$ preserving $c$.

\begin{proposition}\label{proposition:CLOPEN}
Let $\Lambda$ be a locally finite connected graph and $(Q, i)$ be a finite edge-indexed graph. Then the following assertions hold. 
\begin{enumerate}[(1)]
\item Let $c \colon V(\Lambda) \to \mathcal C$ and $c' \colon V(Q) \to \mathcal{C}$ be colorings of the graphs $\Lambda$ and $Q$ respectively. For any closed subset $\mathcal H \subseteq \Sub(\Aut(\Lambda_c))$ consisting of groups acting without inversion, the set 
$$\{H \in \mathcal H \mid H \backslash \Lambda_c \cong (Q_{c'},i)\}$$
is clopen in $\mathcal H$. 
\item Let $c \colon V(\Lambda^{(1)}) \to \{0,1\}$ be the coloring of $\Lambda^{(1)}$ defined by setting $c(v) = 0$ if   $v$ is a  vertex of $\Lambda$ and $c(v) = 1$ if $v$ the midpoint of a geometric edge of $\Lambda$. Let  $c' \colon V(Q) \to \{0,1\}$ be any coloring of $Q$. The set 
$$\{H \in \Sub(\Aut(\Lambda)) \mid H \backslash (\Lambda^{(1)})_c \cong (Q_{c'},i)\}$$
is clopen in $\Sub(\Aut(\Lambda))$. 
\item For any $C > 0$, the set
$$\Sub (\Aut(\Lambda))_{\leq C} := \{H \in \Sub(\Aut(\Lambda)) \mid \# V(H \backslash \Lambda) \leq C\}$$
is clopen in $\Sub(\Aut(\Lambda))$. 
\end{enumerate}
\end{proposition}

\begin{proof}
(1). Let $\mathcal F \subseteq V(\Lambda) \cup E(\Lambda)$ be a finite set of vertices and edges of $\Lambda$. We denote by 
$$\mathcal H_{\text{co-}\mathcal F}$$
the set of those $H \in \mathcal H$ such that $H \mathcal F = V(\Lambda) \cup E(\Lambda)$, i.e.\ those $H \in \mathcal H$ such that $\mathcal F$ meets every $H$-orbit of vertices and every $H$-orbit of edges in $\Lambda$. 

\begin{claim}\label{claim:Hcof}
The set $\mathcal H_{\text{co-}\mathcal F}$ is clopen in $\mathcal H$.
\end{claim}

\begin{claimproof}
Define $\widetilde{\mathcal F}$ as the set consisting of all vertices that are adjacent to a vertex in $\mathcal F$ or incident to an edge in $\mathcal F$. 
Since $\mathcal F$ is finite and $\Lambda$ is locally finite, we infer that $\widetilde{\mathcal F}$ is finite.
Define the set
$$\mathcal J = \{J \in \mathcal H \mid \forall x \in \widetilde{\mathcal F}, \ \exists j \in J : j(x) \in \mathcal F\}.$$
It is clear that $\mathcal H_{\text{co-}\mathcal F} \subseteq \mathcal{J}$, and we claim that $\mathcal H_{\text{co-}\mathcal F} = \mathcal{J}$. Indeed, let $J \in \mathcal J$. Observe that $X = J\mathcal F$ is a subset of $V(\Lambda) \cup E(\Lambda)$ satisfying the property that for any vertex $x$ in $X$, all edges of $\Lambda$ incident to $x$ and all vertices of $\Lambda$ adjacent to $x$ are also in $X$. Since $\Lambda$ is connected, we deduce that $X = V(\Lambda) \cup E(\Lambda)$ and hence that $J \in \mathcal H_{\text{co-}\mathcal F}$.

Now remark that
$$\mathcal H_{\text{co-}\mathcal F} = \mathcal{J} = \bigcap_{x \in \tilde{\mathcal F}} \{J \in \mathcal H \mid J \cap C_x \neq \varnothing\},$$
where $C_x$ is the compact open subset of $\Aut(\Lambda)$ consisting of the elements $h$ with $h(x) \in \mathcal F$. As $\widetilde{\mathcal F}$ is finite, Lemma~\ref{lemma:basicChabauty} ensures that $\mathcal H_{\text{co-}\mathcal F}$ is clopen in $\mathcal{H}$.
\end{claimproof}

\begin{claim}\label{claim:VQ}
The set 
$$\mathcal V_{(Q_{c'}, i), \mathcal F} := \{H \in \mathcal H_{\text{co-}\mathcal F} \mid H \backslash \Lambda_c \cong (Q_{c'},i)\}$$
is clopen in $\mathcal H$.
\end{claim}

\begin{claimproof}
For each $H \in \mathcal H_{\text{co-}\mathcal F}$, the isomorphism type of the edge-indexed (colored) quotient graph $H \backslash \Lambda_c$ is completely determined by the following finite subset of $\mathcal F \times (V(\Lambda) \cup E(\Lambda))$:
$$S_H := \{(x, y) \in {\mathcal F} \times (V(\Lambda) \cup E(\Lambda)) \mid \exists h \in H : hx = y \text{ and } d(y, \mathcal F) \leq 1\}.$$
Moreover, it is clear from Lemma~\ref{lemma:converge} that if $H_n \to H$ in $\mathcal H_{\text{co-}\mathcal F}$ then $S_{H_n} = S_H$ for sufficiently large $n$. Consequently, the set $\mathcal V_{(Q_{c'}, i), \mathcal F}$ is clopen in $\mathcal H_{\text{co-}\mathcal F}$. As $\mathcal H_{\text{co-}\mathcal F}$ is itself clopen in $\mathcal H$ by Claim~\ref{claim:Hcof}, the conclusion follows.
\end{claimproof}

\medskip

We now finish the proof as follows. We must show that the set $\mathcal V_{(Q_{c'},i)} := \{H \in \mathcal H \mid H \backslash \Lambda_c \cong (Q_{c'},i)\}$ is clopen in $\mathcal H$. We may assume that it is nonempty. Fix a base vertex $v_0 \in V(\Lambda)$. For any group $H \in \mathcal V_{(Q_{c'},i)}$, we can find a set of representatives $\mathcal F_0$  of the $H$-orbits of vertices and edges in $\Lambda$, in such a way that $v_0 \in \mathcal F_0$ and that $\mathcal F_0$ is connected.  Notice that there are only finitely many connected subsets $\mathcal F \subseteq V(\Lambda) \cup E(\Lambda)$ containing $v_0$ and such that $\# V(\mathcal F) = \#V(Q)$ and $\# E(\mathcal F) = \# E(Q)$. Let us enumerate all of them, namely $\mathcal F_0, \mathcal F_1, \dots, \mathcal F_m$. We have $\mathcal V_{(Q_{c'},i)} = \bigcup_{j=0}^m \mathcal V_{(Q_{c'},i), \mathcal F_j}$. Each $\mathcal V_{(Q_{c'},i), \mathcal F_j}$ is clopen by Claim~\ref{claim:VQ}, hence $\mathcal V_{(Q_{c'},i)}$ is clopen as well.

\medskip \noindent (2). 
We may identify $\Aut(\Lambda)$ with $\Aut((\Lambda^{(1)})_c)$, which acts without inversion on $\Lambda^{(1)}$. The desired assertion then follows from (1). 

\medskip \noindent (3).
Let $c \colon V(\Lambda^{(1)}) \to \{0,1\}$ be the coloring of $\Lambda^{(1)}$ as defined in (2). For any $C > 0$, there are finitely many edge-indexed (colored) graphs $(Q_{c'}, i)$ that can be isomorphic to $H \backslash (\Lambda^{(1)})_c$ for some $H \in \Sub(\Aut(\Lambda))_{\leq C}$. Moreover, given such a $(Q_{c'}, i)$, if $H' \in \Sub(\Aut(\Lambda))$ satisfies $H' \backslash (\Lambda^{(1)})_c \cong (Q_{c'},i)$ then $\#V(H' \backslash \Lambda) = \#V(H \backslash \Lambda) \leq C$. Indeed, $\#V(H' \backslash \Lambda)$ is equal to the number of vertices $v$ of $Q$ with $c'(v) = 0$. The conclusion then follows from (2).
\end{proof}

\section{The \texorpdfstring{$k$}{k}-closure of a graph automorphism group}\label{section:BEW}

Let $\Lambda$ be a locally finite connected graph. We define the \textbf{$k$-closure} $\closure{k}{J}$ of an automorphism group $J \leq \Aut(\Lambda)$ by $$\closure{k}{J} = \{g \in \Aut(\Lambda) \mid \forall v \in V(\Lambda), \exists h \in J : g|_{B(v,k)} = h|_{B(v,k)}\},$$
where $B(v,k)$ is the ball centered at $v$ and of radius $k$ in $\Lambda$. That notion was first introduced and studied by Banks--Elder--Willis in~\cite{Banks}, in the case where $\Lambda$ is a tree, even though they used the notation $J^{(k)}$ instead of $\closure{k}{J}$.

It is clear from the definition that $\closure{k}{J} \supseteq \closure{\ell}{J} \supseteq J$ for any $k \leq \ell$. Other basic properties of $k$-closures, due to Banks--Elder--Willis, are collected in the following lemma. 

\begin{lemma}\label{lemma:k-closure}
Let $\Lambda$ be a locally finite connected graph. For any $k \geq 0$ and $J \leq \Aut(\Lambda)$, $\closure{k}{J}$ is a closed subgroup of $\Aut(\Lambda)$. Moreover we have
$$\overline J = \bigcap_{k \geq 0} \closure{k}{J}.$$
\end{lemma}

\begin{proof}
The proofs when $\Lambda$ is a locally finite tree are given in \cite{Banks}*{Proposition~3.4}, but they are independent from the tree structure and thus also work for any locally finite connected graph $\Lambda$.
\end{proof}

In view of Lemma~\ref{lemma:inter} (1), the previous lemma implies that $\closure{k}{J} \to \overline{J}$ in $\Sub({\Aut(\Lambda)})$. The next result is then a key tool for the proof of Theorem~\ref{theorem:simpleclosed}. In order to facilitate its statement, we introduce the following notation. Given a group $\Gamma \leq \Aut(\Lambda)$, we write
$$\Sub(\Aut(\Lambda))_{\geq \Gamma} = \{H \in \Sub(\Aut(\Lambda)) \mid H \geq \tau \Gamma \tau^{-1} \text{ for some } \tau \in \Aut(\Lambda)\}.$$
Observe that if the normalizer of $\Gamma$ in $\Aut(\Lambda)$ is cocompact, then $\Sub(\Aut(\Lambda))_{\geq \Gamma}$ is Chabauty-closed. Given a group $H \leq \Aut(\Lambda)$, a vertex $v \in V(\Lambda)$ and an integer $r \geq 0$, we also write $H_v^{[r]}$ for the pointwise stabilizer of the ball $B(v,r)$ in $H$.

\begin{proposition}\label{proposition:convergence_k}
Let $\Lambda$ be a locally finite connected graph, $\Gamma \leq \Aut(\Lambda)$ act cocompactly on $\Lambda$ and $H \in \Sub(\Aut(\Lambda))_{\geq \Gamma}$. Fix $v_0 \in V(\Lambda)$. Then for each $k \geq 0$, the set
$$V_k := \{J \in \Sub(\Aut(\Lambda))_{\geq \Gamma} \mid \sigma J \sigma^{-1} \leq \closure{k}{H} \text{ for some } \sigma \in \Aut(\Lambda)_{v_0}^{[k]} \}$$
is a neighborhood of $H$ in $\Sub(\Aut(\Lambda))_{\geq \Gamma}$.
\end{proposition}

\begin{proof}
Consider a sequence $H_n \to H$ in $\Sub(\Aut(\Lambda))_{\geq \Gamma}$ and let us show that $H_n \in V_k$ for sufficiently large $n$. Let $X \subset \Lambda$ be a compact fundamental domain for the action of $\Gamma$ on $\Lambda$. For each $n$, let $\tau_n \in \Aut(\Lambda)$ be such that $H_n \geq \tau_n \Gamma \tau_n^{-1}$. We may assume, up to precomposing $\tau_n$ with an adequate element of $\Gamma$, that $\tau_n$ sends $v_0$ to a vertex in $X$. Since $X$ is compact, the sequence $(\tau_n)$ is bounded and we can further assume (by passing to a subsequence) that $(\tau_n)$ converges to some $\tau \in \Aut(\Lambda)$. Define $\sigma_n := \tau \tau_n^{-1}$ for each $n \geq 1$ so that $\sigma_n \to \id$.
In this way, we have
$$\Gamma' := \tau \Gamma \tau^{-1} \leq \tau \tau_n^{-1} H_n \tau_n \tau^{-1} = \sigma_n H_n \sigma_n^{-1} =: H'_n \quad \text{for each $n \geq 1$},$$
where $\Gamma'$ also acts cocompactly on $\Lambda$ with $X' := \tau(X)$ as a fundamental domain. Moreover, as $\sigma_n \to \id$, we have $H'_n \to H$ by Lemma~\ref{lemma:conjugate} and in particular $\Gamma' \leq H$.

In order to conclude, it suffices to find $N \geq 1$ such that $H'_n \leq \closure{k}{H}$ for each $n \geq N$. Let $D$ be the diameter of $X'$ and set
$$K := \{g \in \Aut(\Lambda) \mid d(g(v_0), v_0) \leq 2D\}$$
and
$$U := \Aut(\Lambda)_{v_0}^{[k+D]}.$$
The set $K$ is compact and the set $U$ is open, so there exists $N \geq 1$ such that $H'_n \in \mathcal{V}_{K,U}(H)$ for each $n \geq N$. In particular, we have $H'_n \cap K \subseteq HU$ for each $n \geq N$. This exactly means that, for any $g \in H'_n$ (with $n \geq N$) satisfying $d(g(v_0),v_0) \leq 2D$, there exists $h \in H$ such that $g|_{B(v_0,k+D)} = h|_{B(v_0,k+D)}$.

We need to show that $H'_n \leq \closure{k}{H}$ for each $n \geq N$. In order to do so, consider $g \in H'_n$ with $n \geq N$ and $v \in V(\Lambda)$. Let $\gamma_1 \in \Gamma'$ be such that $d(\gamma_1 g(v), v_0) \leq D$ and $\gamma_2 \in \Gamma'$ be such that $d(\gamma_2(v_0), v) \leq D$. Those elements exist because $D$ is the diameter of the fundamental domain $X'$ for the action of $\Gamma'$. The two previous inequalities imply that $d(\gamma_1g\gamma_2(v_0),v_0) \leq 2D$. Hence, by definition of $N$ there exists $h \in H$ with
$$\gamma_1g\gamma_2|_{B(v_0,k+D)} = h|_{B(v_0,k+D)},$$
which is equivalent to saying that
$$g|_{B(\gamma_2(v_0),k+D)} = \gamma_1^{-1} h \gamma_2^{-1}|_{B(\gamma_2(v_0),k+D)}.$$
But $d(\gamma_2(v_0), v) \leq D$, so $B(\gamma_2(v_0), k+D) \supseteq B(v,k)$ and
$$g|_{B(v,k)} = \gamma_1^{-1} h \gamma_2^{-1}|_{B(v,k)},$$
which is sufficient to conclude since $\gamma_1^{-1} h \gamma_2^{-1} \in H$.
\end{proof}

The following observation describes a local algebraic property that is preserved when taking the $k$-closure (with a sufficiently large $k$).

\begin{proposition}\label{proposition:locallyProPi}
Let $\Lambda$ be a locally finite connected graph and $H \in \Sub(\Aut(\Lambda))$. Let also $\pi$ be a set of primes and $r \geq 0$. Suppose that $H_v^{[r]}$ is a pro-$\pi$ group for all $v \in V(\Lambda)$. Then for each $k \geq r+1$, the group $(\closure{k}{H})_v^{[r]}$ is a pro-$\pi$ group for all $v \in V(\Lambda)$.

In particular, if $H$ acts cocompactly on $\Lambda$ and has an open pro-$\pi$ subgroup, then so does $\closure{k}{H}$ for all sufficiently large $k$. 
\end{proposition}

\begin{proof}
Since $\closure{k}{H} \leq \closure{\ell}{H}$ for all $k \geq \ell$ and since a closed subgroup of a pro-$\pi$ group is pro-$\pi$, it suffices to consider $G = \closure{r+1}{H}$. We show that for each $n \geq r$ and each $v \in V(\Lambda)$, the finite group $\bigslant{G_v^{[n]}}{G_v^{[n+1]}}$ is a $\pi$-group. This assertion implies the required conclusion.

Fix $n \geq r$ and $v \in V(\Lambda)$ and assume for a contradiction that $\bigslant{G_v^{[n]}}{G_v^{[n+1]}}$ is not a $\pi$-group. Then it contains an element $g$ of prime order $p$, with $p \not \in \pi$. There exists a vertex $z$ with $d(v,z) = n$ such that the restriction $g|_{B(z,1)}$ contains a $p$-cycle. Let $x$ be a vertex on a geodesic path from $v$ to $z$ such that $d(x,z) = r$. Thus $g$ fixes $B(x,r) \subseteq B(v,n)$ pointwise. Since $g \in G = \closure{r+1}{H}$, there is $h \in H$ such that $g|_{B(x,r+1)} = h|_{B(x,r+1)}$. Hence $h$ belongs to $H_x^{[r]}$ and the image of $h$ modulo $H_x^{[r+1]}$ is of order $p \not \in \pi$. This contradicts the hypothesis that $H_x^{[r]}$ is pro-$\pi$.

Now suppose $H$ acts cocompactly on $\Lambda$ and has an open pro-$\pi$ subgroup $U$. Since $U$ is open, there exists $v_0 \in V(\Lambda)$ and $r \geq 0$ such that $H_{v_0}^{[r]} \subseteq U$. Let $X \ni v_0$ be a compact fundamental domain for the action of $H$ on $\Lambda$, and denote by $D$ its diameter. For each vertex $x \in X$, we have
$$H_x^{[r+D]} \subseteq H_{v_0}^{[r]} \subseteq U,$$
so $H_x^{[r+D]}$ is a pro-$\pi$ group. Since $X$ is a fundamental domain for the action of $H$ on $\Lambda$, we even have that $H_v^{[r+D]}$ is a pro-$\pi$ group for all $v \in V(\Lambda)$. By the previous assertion, this implies that $\closure{k}{H}$ has an open pro-$\pi$ subgroup for each $k \geq r+D+1$. 
\end{proof}

Applying the previous proposition in the case of the empty set of primes, we obtain the following corollary for discrete groups. 

\begin{corollary}\label{corollary:discrete}
Let $\Lambda$ be a locally finite connected graph and $H$ be a discrete subgroup of $\Aut(\Lambda)$ acting cocompactly on $\Lambda$. Then $H = \closure{k}{H}$ for all sufficiently large $k$.
\end{corollary}

\begin{proof}
Applying Proposition~\ref{proposition:locallyProPi} to the empty set $\pi = \varnothing$, we obtain that $\closure{k}{H}$ is discrete for each sufficiently large $k$. Since $H$ acts cocompactly on $\Lambda$, so does $\closure{k}{H}$ for any $k$. Fixing $k_0$ such that $\closure{k_0}{H}$ is discrete, we deduce that the index of $H$ in $\closure{k_0}{H}$ is finite. Since $H \leq \closure{k+1}{H} \leq \closure{k}{H} \leq \closure{k_0}{H}$ for each $k \geq k_0$, the conclusion follows from Lemma~\ref{lemma:k-closure}. 
\end{proof}

\section{Finite quotients of groups acting on graphs}

The goal of this section is to prove Theorem~\ref{theorem:simpleclosed}. 

We first recall that, for a topological group $G$, the symbol $G^{(\infty)}$ denotes the intersection of all open subgroups of finite index of $G$. The following lemma is classical. The notation $P \leq_{\text{ofi}} G$ means that $P$ is an open subgroup of finite index of $G$.

\begin{lemma}\label{lemma:normal}
Let $G$ be a topological group and $P \leq_{\text{ofi}} G$. There exists $R \leq P$ such that $R \unlhd_{\text{ofi}} G$. In particular, $G^{(\infty)}$ coincides with the intersection of all open normal subgroups of finite index of $G$.
\end{lemma}

\begin{proof}
It suffices to take for $R$ the kernel of the natural action of $G$ on $G/P$.
\end{proof}

The next result shows, in the context of automorphism groups of graphs, how $k$-closures preserve open subgroups of finite index.

\begin{lemma}\label{lemma:ofi}
Let $\Lambda$ be a locally finite connected graph and $H \in \Sub(\Aut(\Lambda))$ act cocompactly on $\Lambda$. If $P \leq_{\text{ofi}} H$, then $[\closure{k}{H} : \closure{k}{P}] \leq [H : P]$ for all sufficiently large $k$.
\end{lemma}

\begin{proof}
Fix $v_0 \in V(\Lambda)$ and let $m = [H : P]$. We can write
$$H = \bigsqcup_{i=1}^m h_i P$$
for some $h_1, \ldots, h_m \in H$. Since $H$ acts cocompactly on $\Lambda$, the action of $P$ on $\Lambda$ is also cocompact. Let $X \subseteq \Lambda$ be a compact fundamental domain for the action of $P$ and denote by $D$ the diameter of $X$.

The fact that $P$ is an open subgroup of $H$ implies that there exists $R \geq 0$ with
$$H_{v_0}^{[R]} \subseteq P.$$
We claim that $[\closure{k}{H} : \closure{k}{P}] \leq m$ for each $k \geq R+D+1$. To prove the claim, we fix $k \geq R+D+1$ and show that
$$\closure{k}{H} = \bigcup_{i=1}^m h_i \closure{k}{P}.$$
Take $g \in \closure{k}{H}$ and $v \in V(\Lambda)$. There exists $i \in \{1, \ldots, m\}$ and $x \in P$ such that $g |_{B(v, k)} = h_i x |_{B(v, k)}$, which is equivalent to saying that $h_i^{-1} g |_{B(v, k)} = x |_{B(v, k)}$.
If we prove that $i$ is independent of the choice of $v$, then we will get $h_i^{-1} g \in \closure{k}{P}$ which will end the proof. Since $\Lambda$ is connected, it suffices to show that the value of $i$ is the same for any two adjacent vertices. Fix $v$ and $v'$ two neighboring vertices of $\Lambda$ and suppose that
$$g |_{B(v, k)} = h_i x |_{B(v, k)} \quad \text{and} \quad g |_{B(v', k)} = h_j y |_{B(v', k)}$$
for some $x, y \in P$ and some $i, j \in \{1, \ldots, m\}$. It follows that
$$h_i x |_{B(v, k-1)} = h_j y |_{B(v, k-1)}$$
or equivalently that
$$h_j^{-1} h_i x y^{-1} |_{B(y(v), k-1)} = \id |_{B(y(v), k-1)}.$$
The element $e := h_j^{-1} h_i x y^{-1}$ is thus such that $e \in H_{y(v)}^{[k-1]}$. As $X$ is a fundamental domain (with diameter $D$) for the action of $P$ on $\Lambda$, there exists $p \in P$ such that $p(y(v)) \in B(v_0, D)$. Hence, the element $pep^{-1}$ satisfies
$$pep^{-1} \in H_{p(y(v))}^{[k-1]} \subseteq H_{v_0}^{[k-1-D]} \subseteq H_{v_0}^{[R]} \subseteq P.$$
We get $h_j^{-1} h_i x y^{-1} = e \in P$ and thus $h_{j}^{-1} h_i \in P$, which implies that $i = j$ as desired.
\end{proof}

Before proving Theorem~\ref{theorem:simpleclosed}, we still need a technical lemma.

\begin{lemma}\label{lemma:index}
Let $\Lambda$ be a locally finite connected graph and $H_n \to H$, $L_n \to L$ be two converging sequences in $\Sub(\Aut(\Lambda))$ such that $L_n \leq H_n$ for each $n \geq 1$. Assume that there exists $C > 0$ such that $H_n \in \Sub(\Aut(\Lambda))_{\leq C}$ for each $n \geq 1$. Suppose also that there exists $S \geq 1$ such that $[H_n : L_n] \leq S$ for each $n \geq 1$. Then $L \leq H$ and $[H : L] \leq S$.
\end{lemma}

\begin{proof}
The fact that $L \leq H$ is clear. For each $n \geq 1$, let $F_n \subseteq \Aut(\Lambda)$ be such that $H_n = L_n F_n$ and $|F_n| \leq S$. We directly get that $L_n \in \Sub(\Aut(\Lambda))_{\leq CS}$ for each $n \geq 1$. If $v_0 \in V(\Lambda)$ is a fixed vertex, for each $n \geq 1$ and $f \in F_n$ we can thus assume that $d(f(v_0),v_0) \leq CS$. By adding elements to $F_n$ if necessary, we can also suppose that $|F_n| = S$ and write $F_n = \{f_1^{(n)}, \ldots, f_S^{(n)}\}$. Since the set $\{g \in \Aut(\Lambda) \mid d(g(v_0),v_0) \leq CS\}$ is compact, we can finally assume by passing to subsequences that $(f_i^{(n)})$ converges to some $f_i \in \Aut(\Lambda)$ for each $i \in \{1, \ldots, S\}$. Define $F := \{f_1, \ldots, f_S\}$. It is clear that $F \subseteq H$ and we claim that $H = LF$. Take $h \in H$. By Lemma~\ref{lemma:converge}, there exists a converging sequence $h_n \to h$ with $h_n \in H_n$ for each $n \geq 1$. As $H_n = L_n F_n$, we can write $h_n = \ell_n f_{i_n}^{(n)}$ with $\ell_n \in L_n$ and $i_n \in \{1,\ldots,S\}$. There is a subsequence $(i_{k(n)})$ of $(i_n)$ which is constant, say equal to $j \in \{1, \ldots, S\}$. Then $h_{k(n)} = \ell_{k(n)} f_j^{(n)}$ and hence $\ell_{k(n)} = h_{k(n)} (f_j^{(n)})^{-1} \to hf_j^{-1}$. This limit belongs to $L$, so $hf_j^{-1} = \ell \in L$ and $h = \ell f_j$.
\end{proof}

The proof of Theorem~\ref{theorem:simpleclosed} is now an easy combination of the previous results.

\begin{proof}[Proof of Theorem~\ref{theorem:simpleclosed}]
Let $S = \limsup_{n\to \infty}\, [H_n : H_n^{(\infty)}]$. Without loss of generality, we may assume that $ [H_n : H_n^{(\infty)}]\leq S$ for each $n \geq 1$.
By Proposition~\ref{proposition:convergence_k}, we may further assume that for each $k \geq 0$, there exists $N(k) \geq 1$ such that $H_n \leq \closure{k}{H}$ for each $n \geq N(k)$. In order to prove that $[H : H^{(\infty)}] \leq S$, it suffices to prove that $[H : P] \leq S$ for each $P \leq_{\text{ofi}} H$. By Lemma~\ref{lemma:ofi}, there exists $K \geq 0$ such that $\closure{k}{P} \leq_{\text{ofi}} \closure{k}{H}$ for any $k \geq K$. Let us temporarily fix $k \geq K$. For each $n \geq N(k)$, we have $H_n \leq \closure{k}{H}$ and hence $\closure{k}{P} \cap H_n \leq_{\text{ofi}} H_n$. By hypothesis, this means that $[H_n : \closure{k}{P} \cap H_n] \leq S$. Letting $n$ tend to infinity, we obtain with Lemma~\ref{lemma:index} that $[H : \closure{k}{P} \cap H] \leq S$ for each $k \geq K$. Now letting $k$ tend to infinity and because $\closure{k}{P} \to \overline{P}$ (see Lemmas~\ref{lemma:inter} (1) and~\ref{lemma:k-closure}), we get $[H : \overline{P}] \leq S$. An open subgroup is always closed, so $\overline{P} = P$ and the conclusion follows.
\end{proof}

\section{Trees}

\subsection{Existence and conjugation of tree lattices}\label{subsection:trees}

When $\Lambda$ is a locally finite tree, Propositions~\ref{proposition:existence} and~\ref{proposition:conjugate} below (which come from~\cite{Bass-Kulkarni} and~\cite{Bass} respectively) can be used to drop the hypothesis about $\Gamma$ in Theorem~\ref{theorem:simpleclosed}.

\begin{proposition}\label{proposition:existence}
Let $T$ be a locally finite tree. Let $H \leq \Aut(T)$ act cocompactly on $T$ and suppose that $\overline{H}$ is unimodular. Then $H$ contains a free uniform lattice, i.e.\ there exists a  discrete subgroup $\Gamma \leq H$ acting freely and cocompactly on $T$.
\end{proposition}

\begin{proof}
See~\cite{Bass-Kulkarni}*{Existence Theorem}.
\end{proof}

Given a tree $T$ and two groups $H, H' \leq \Aut(T)$ acting without inversion on $T$, we write the equality
$$H \backslash T = H' \backslash T$$ 
whenever $H$ and $H'$ have the same orbits on $T$. The latter condition, which means that the canonical projections $p \colon T \to H \backslash T$ and $p' \colon T \to H' \backslash T$ coincide, implies in particular that the quotient graphs $H \backslash T$ and $H' \backslash T$ are isomorphic \emph{as edge-indexed graphs}, since the edge-indexing function of the quotient graph is completely determined by the projection map. The following basic fact clarifies the difference between isomorphism and equality of quotients. 
 
\begin{lemma}\label{lem:Lifting}
Let $T$ be a tree and $H, H' \leq \Aut(T)$ act without inversion on $T$. If $H \backslash T $ and $ H' \backslash T$ are isomorphic as edge-indexed graphs, then there exists $g \in \Aut(T)$ such that $H \backslash T = H'' \backslash T$, where $H'' = g H' g^{-1}$.
\end{lemma}

\begin{proof}
This is a particular case of \cite{Mosher}*{Lemma~13 (1)}.
\end{proof}

\begin{proposition}\label{proposition:conjugate}
Let $T$ be a tree and $H, H' \leq \Aut(T)$ act without inversion on $T$. Suppose that $H \backslash T = H' \backslash T$. If $\Gamma \leq H$ acts freely on $T$, then there exists $\tau \in \Aut(T)$ such that $\tau\Gamma\tau^{-1} \leq H'$.
\end{proposition}

\begin{proof}
See~\cite{Bass}*{Corollary~5.3}.
\end{proof}

\begin{corollary}\label{corollary:tree:common-lattice}
Let $T$ be a locally finite tree and $H \in \Sub(\Aut(T))$. Suppose that $H$ is unimodular and acts cocompactly on $T$ (these conditions hold, for instance, if $H$ is edge-transitive and type-preserving). Then $H$ has a subgroup $\Gamma$ acting freely and cocompactly on $T$ such that $\Sub(\Aut(T))_{\geq \Gamma}$ is a neighborhood of $H$ in $\Sub(\Aut(T))$.
\end{corollary}

\begin{proof}
Upon replacing $T$ by its first barycentric subdivision, we may assume that $H$ acts without inversion. By Proposition~\ref{proposition:existence}, there exists $\Gamma \leq H$ acting cocompactly and freely on $T$. Consider a converging sequence $H_n \to H$ in $\Sub(\Aut(T))$. By Proposition~\ref{proposition:CLOPEN} (3) $H_n$ acts with at most $C$ orbits of vertices  for all sufficiently large $n$, where $C = \# V(H \backslash T)$. We may then deduce from  Lemma~\ref{lemma:basicChabauty} that $H_n$ acts without inversion  for all sufficiently large $n$ because $H$ does. Moreover Proposition~\ref{proposition:CLOPEN} (1) ensures that, for sufficiently large $n$, the quotient graphs $H_n \backslash T$ and $H \backslash T$ are isomorphic as edge-indexed graphs. Hence, by Lemma~\ref{lem:Lifting} and Proposition~\ref{proposition:conjugate}, for sufficiently large $n$, there exists $\tau_n \in \Aut(T)$ such that $\tau_n \Gamma \tau_n^{-1} \leq H_n$, i.e.\ $H_n \in \Sub(\Aut(T))_{\geq \Gamma}$.
\end{proof}

We record the following result for its own interest. It shows that, in Corollary~\ref{corollary:tree:common-lattice}, the choice of $\Gamma$  can be made uniform, i.e.\ independent of the choice of $H$. In order to make this precise, we define
$$\Sub(\Aut(T))_{\leq C}^{0} := \{H \in \Sub(\Aut(T))_{\leq C} \mid H \text{ is unimodular}\}.$$

\begin{corollary}\label{corollary:tree:common-lattice2}
Let $T$ be a locally finite tree. For each $C>0$, the set $\Sub(\Aut(T))_{\leq C}^{0}$ is clopen in $\Sub(\Aut(T))$. Moreover there exists a subgroup $\Gamma \leq \Aut(T)$ acting freely and cocompactly on $T$   such that 
$$\Sub(\Aut(T))_{\leq C}^{0} \subseteq \Sub(\Aut(T))_{\geq \Gamma}.$$
\end{corollary}

\begin{proof}
We already know by Proposition~\ref{proposition:CLOPEN} (3) that $\Sub(\Aut(T))_{\leq C}$ is a clopen subset of $\Sub(\Aut(T))$. Moreover the set $\Sub(\Aut(T))^{0}$ of unimodular subgroups is closed by Theorem~\ref{theorem:Unimod}. In particular $\Sub(\Aut(T))_{\leq C}^{0}$ is closed. 

For each $H \in \Sub(\Aut(T))_{\leq C}^{0}$,   Corollary~\ref{corollary:tree:common-lattice} yields a discrete cocompact group $\Gamma$ such that $\Sub(\Aut(T))_{\geq \Gamma}$ is a neighborhood of $H$. Since every locally compact group containing a lattice is unimodular, this implies that $\Sub(\Aut(T))_{\leq C}^{0}$ is also open. 

Let us now partition the set $\Sub(\Aut(T))_{\leq C}^{0}$ into subsets $\mathcal V_1, \dots, \mathcal V_m$ in such a way that $H, H' \in \mathcal V_i$ if and only if $H\backslash T \cong H' \backslash T$ as edge-indexed graphs. By Lemma~\ref{lem:Lifting} and Proposition~\ref{proposition:conjugate}, for each $i$ there exists a discrete cocompact group $\Gamma_i$ such that $\mathcal V_i \subseteq  \Sub(\Aut(T))_{\geq \Gamma_i}$. In particular
$$\Sub(\Aut(T))_{\leq C}^{0} \subseteq \bigcup_{i=1}^m \Sub(\Aut(T))_{\geq \Gamma_i}.$$
By \cite{Bass-Kulkarni}*{Commensurability Theorem}, upon replacing each $\Gamma_i$ by a conjugate, we may assume that they are pairwise commensurate, i.e.\ the index of $\Gamma_i \cap \Gamma_j$ is of finite index in $\Gamma_i$ for all $i$ and $j$. It follows that $\Gamma = \bigcap_{i=1}^m \Gamma_i$ is itself a cocompact lattice in $\Aut(T)$. The required assertion follows since 
$$\Sub(\Aut(T))_{\geq \Gamma_i} \subseteq \Sub(\Aut(T))_{\geq \Gamma}$$
 for all $i$. 
\end{proof}

We then deduce the following corollary from Proposition~\ref{proposition:convergence_k}.

\begin{corollary}\label{corollary:tree:convergence<k-closure}
Let $T$ be a locally finite tree and $H \in \Sub(\Aut(T))$. Suppose that $H$ is unimodular and acts cocompactly on $T$. Fix $v_0 \in V(T)$. Then for each $k \geq 0$, the set
$$\{J \in \Sub(\Aut(T)) \mid \sigma J \sigma^{-1} \leq \closure{k}{H} \text{ for some } \sigma \in \Aut(T)_{v_0}^{[k]}\}$$
is a neighborhood of $H$ in $\Sub(\Aut(T))$.
\end{corollary}

\begin{proof}
This is the combination of Corollary~\ref{corollary:tree:common-lattice} and Proposition~\ref{proposition:convergence_k}.
\end{proof}

Arguing similarly, we obtain the following consequence of Theorem~\ref{theorem:simpleclosed}.

\begin{corollary}\label{corollary:tree}
Let $T$ be a locally finite tree and $H_n \to H$ be a converging sequence in $\Sub({\Aut(T)})$. Suppose that $H$ is unimodular and acts cocompactly on $T$. Then we have
$$[H : H^{(\infty)}] \leq \limsup_{n\to \infty}\, [H_n : H_n^{(\infty)}] .$$

In particular, if $H_n$ has no proper open subgroup of finite index for each $n \geq 1$ then $H$ has no proper open subgroup of finite index.
\end{corollary}

\begin{proof}
Let $\Gamma$ be the subgroup of $H$ given by Corollary~\ref{corollary:tree:common-lattice}. It acts cocompactly on $T$ and is such that $\Sub(\Aut(T))_{\geq \Gamma}$ is a neighborhood of $H$. We thus have $H_n \in \Sub(\Aut(T))_{\geq \Gamma}$ for all sufficiently large $n$, and the conclusion follows from Theorem~\ref{theorem:simpleclosed}.
\end{proof}

\subsection{Limits of simple groups acting on trees}

The goal of this section is to prove the next theorem, which is a stronger version of Theorem~\ref{theorem:LimitsOfSimple}. 

\begin{theorem}\label{theorem:LimitsOfSimple2}
Let $T$ be a locally finite tree all of whose vertices have degree~$\geq 2$. For any $C > 0$, the Chabauty-closure of the set of abstractly simple groups in $\Sub(\Aut(T))_{\leq C}$ is the set of groups in $\Sub(\Aut(T))_{\leq C}$ without proper open subgroup of finite index.
\end{theorem}

We start by proving the following.

\begin{proposition}\label{proposition:BoundedCovol}
Let $T$ be a locally finite tree and let $C > 0$. The set 
$$\Sub(\Aut(T))_{\leq C}^{(\infty)} := \{H \in \Sub(\Aut(T))_{\leq C} \mid H = H^{(\infty)}\}$$
is closed in $\Sub({\Aut(T)})$. 
\end{proposition}

\begin{proof}
Let $H_n \to H$ be a converging sequence in $\Sub(\Aut(T))$ with $H_n \in \Sub(\Aut(T))_{\leq C}^{(\infty)}$ for each $n$. We already know by Proposition~\ref{proposition:CLOPEN} (3) that $H \in \Sub(\Aut(T))_{\leq C}$. For each $n \geq 1$, $H_n$ acts cocompactly on $T$ and is thus compactly generated. Therefore, the image of the modular character of $H_n$ is a finitely generated subgroup of $\R$, which is thus residually finite. In particular, the condition that $H_n = H_n^{(\infty)}$ implies that $H_n$ is unimodular. By Theorem~\ref{theorem:Unimod}, $H$ is also unimodular and we can apply Corollary~\ref{corollary:tree} to get $H = H^{(\infty)}$, as required. This confirms that $\Sub(\Aut(T))_{\leq C}^{(\infty)}$ is closed.
\end{proof}

There remains to show that any group in $\Sub(\Aut(T))_{\leq C}^{(\infty)}$ is a limit of abstractly simple groups in that same set. Before proving this we need two more technical results.

\begin{lemma}\label{lemma:invariant-subtree}
Let $T$ be a locally finite tree all of whose vertices have degree~$\geq 2$ and $H \leq \Aut(T)$ be a closed subgroup without any infinite cyclic discrete quotient (e.g. $H = H^{(\infty)}$). If $H$ acts cocompactly on $T$, then it does not preserve any proper non-empty subtree and does not fix any end of $T$.
\end{lemma}

\begin{proof}
Since all vertices of $T$ have degree~$\geq 2$ and $H$ acts cocompactly on $T$, we deduce from~\cite{Tits_arbre}*{Lemme~4.1} that $H$ does not preserve any non-empty subtree of $T$.

Suppose now for a contradiction that $H$ fixes some end $b \in \bd T$. Let $(v_n)$ be the sequence of vertices on a ray in $T$ toward $b$. Then the map $\phi \colon H \to \Z$ defined by $\phi(h) := \lim_{n \to \infty} d(h(v_n), v_n)$ is a group homomorphism and has infinite image (because $H$ acts cocompactly on $T$), which contradicts the fact that $H$ has no infinite cyclic discrete quotient. 
\end{proof}

In the following proposition and as in \cite{Banks}, given $J \leq \Aut(T)$ and $k > 0$, the symbol $J^{+_k}$ denotes the subgroup of $J$ generated by the pointwise stabilizers of $(k-1)$-balls around edges of $T$.

\begin{proposition}\label{proposition:BEW}
Let $T$ be a locally finite tree and $G \leq \Aut(T)$ be a non-discrete group which acts cocompactly on $T$, does not preserve any proper non-empty subtree and does not fix any end of $T$. Suppose that $G = \closure{k}{G}$ for some $k \geq 0$. Then $G^{+_k}$ is abstractly simple and $G/G^{+_k}$ is virtually free.
\end{proposition}

\begin{proof}
From \cite{Banks}*{Theorem~7.3} we know that $G^{+_k}$ is abstractly simple or trivial. Also, it is clear from the definition that $G^{+_k}$ is an open normal subgroup of $G$. Since $G$ is non-discrete, $G^{+_k}$ is non-discrete and in particular non-trivial (hence simple).

The discrete quotient group $G/G^{+_k}$ acts cocompactly on the quotient graph $G^{+_k}\backslash T$. Bass--Serre theory ensures that $G^{+_k}$ is the fundamental group of a graph of groups, whose underlying graph is nothing but $G^{+_k} \backslash T$ (see~\cite{Serre}*{\S I.5.4, Th\'eor\`eme 13}). By definition  $G^{+_k}$ is generated by pointwise stabilizers of edges. In particular it is generated   by   vertex stabilizers. It then follows that the quotient graph $G^{+_k} \backslash T$ is a tree (see~\cite{Serre}*{\S I.5.4, Corollaire 1}). We next observe that the $G/G^{+_k}$-action on the tree $G^{+_k} \backslash T$ is proper. Indeed, a coset $g G^{+_k}$ stabilizes a vertex in $G^{+_k} \backslash T$ if and only if $gv \in G^{+_k} v$ for some $v \in V(T)$. This is equivalent to the requirement that $g \in G^{+_k} U$, where $U$ is the stabilizer of $v$ in $G$, which is compact. This confirms that the stabilizer of a vertex of $G^{+_k} \backslash T$ in the discrete quotient group $G/G^{+_k}$ is indeed compact, hence finite. Therefore $G/G^{+_k}$ is a discrete group acting properly and cocompactly on a tree. It is thus virtually free.
\end{proof}

\begin{proposition}\label{proposition:DenseSimple}
Let $T$ be a locally finite tree all of whose vertices have degree~$\geq 2$ and let $C >0$. In $\Sub(\Aut(T))_{\leq C}^{(\infty)}$,
the subset consisting of the abstractly simple groups is dense. 
\end{proposition}

\begin{proof}
Pick any $H \in \Sub(\Aut(T))_{\leq C}^{(\infty)}$. We must show that $H$ is a limit of abstractly simple groups contained in $\Sub(\Aut(T))_{\leq C}^{(\infty)}$. For each $k > 0$, set $H_k = (\closure{k}{H})^{+_k}$. First note that $H$ is not discrete, otherwise it would be virtually free, hence residually finite, contradicting $H = H^{(\infty)}$. We can therefore invoke Lemma~\ref{lemma:invariant-subtree} and Proposition~\ref{proposition:BEW} (applied to $\closure{k}{H}$) to get that $H_k$ is abstractly simple and $\closure{k}{H}/H_k$ is virtually free.

Since $\closure{k}{H}/H_k$ is virtually free, it is residually finite. Recalling now that $H$ has no finite discrete quotient other than the trivial one, we infer that $H$ has trivial image in $\closure{k}{H}/H_k$, so that $H \leq H_k \leq \closure{k}{H}$. Since $\closure{k}{H} \to H$ (by Lemma~\ref{lemma:k-closure} and Lemma~\ref{lemma:inter} (1)), we also get that $H_k \to H$, thereby completing the proof. 
\end{proof}

\begin{proof}[Proof of Theorem~\ref{theorem:LimitsOfSimple2}]
Follows by assembling Propositions~\ref{proposition:BoundedCovol} and~\ref{proposition:DenseSimple}.
\end{proof}

\begin{remark}\label{rem:BuMo}
It is important to note that the set 
$$\{H \in \Sub({\Aut(T)}) \mid H \text{ is locally $2$-transitive and } H = H^{(\infty)}\}$$ 
(in the terminology of \cite{BM}) may contain groups that are not topologically simple. Explicit examples of such $H$ are constructed in \cite{BM}*{Example~1.2.1}, where $T$ is regular of degree $p^2 + p +1$ ($p$ being an arbitrary prime). In particular, the set of topologically simple locally $2$-transitive closed subgroups of $\Aut(T)$ is generally not Chabauty-closed.
\end{remark}

The following result shows that the conclusion of Theorem~\ref{theorem:LimitsOfSimple} may fail if the tree $T$ is allowed to have vertices of degree~$1$.

\begin{lemma}\label{lemma:Valency1}
Let $T$ be the universal covering tree of the graph on $7$ vertices depicted in the figure below. Let $V_1, V_3$ and $V_8$ denote the set of vertices of $T$ of degree $1$, $3$ and~$8$ respectively. Let $X$ be the subtree of $T$ which is the convex hull of $V_3$. Thus $X$ is isomorphic to the trivalent tree. Its vertex set is $V_3 \cup V_8$, and those two sets $V_3$ and $V_8$ are the two parts in the canonical bipartition of $X$. The following assertions hold.
\begin{enumerate}[(1)]
\item $\Aut(T)$ has a closed subgroup $H$ isomorphic to 
$$\left(\prod_{v \in V_8} \Alt(5)\right)\rtimes \Aut(X)^+.$$ 
\item $H = H^{(\infty)}$. 
\item $H$ is not a Chabauty limit of topologically simple closed subgroups of $\Aut(T)$.
\end{enumerate}
\end{lemma}

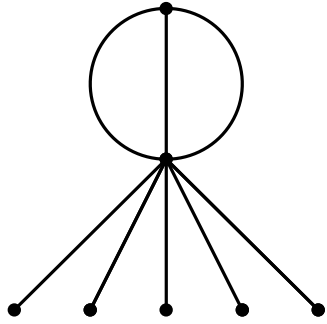
\begin{center}
\psset{xunit=2.0cm,yunit=2.0cm,algebraic=true,dimen=middle,dotstyle=o,dotsize=5pt 0,linewidth=0.8pt,arrowsize=3pt 2,arrowinset=0.25}
\begin{pspicture*}(-1.2,-1.2)(1.2,1.2)
\psline[linewidth=1.2pt](0.,1.)(0.,0.)
\psline[linewidth=1.2pt](0.,0.)(-1.,-1.)
\psline[linewidth=1.2pt](0.,0.)(-0.5,-1.)
\psline[linewidth=1.2pt](0.,0.)(1.,-1.)
\psline[linewidth=1.2pt](0.,0.)(0.5,-1.)
\parametricplot[linewidth=1.2pt]{1.5707963267948966}{4.71238898038469}{1.*0.5*cos(t)+0.*0.5*sin(t)+0.|0.*0.5*cos(t)+1.*0.5*sin(t)+0.5}
\parametricplot[linewidth=1.2pt]{1.5707963267948966}{4.71238898038469}{-1.*0.5*cos(t)+0.*0.5*sin(t)+0.|0.*0.5*cos(t)+1.*0.5*sin(t)+0.5}
\psline[linewidth=1.2pt](0.,0.)(1.,-1.)
\psline[linewidth=1.2pt](0.,0.)(-0.5,-1.)
\psline[linewidth=1.2pt](0.,0.)(0.,-1.)
\begin{scriptsize}
\psdots[dotstyle=*](0.,0.)
\psdots[dotstyle=*](0.,1.)
\psdots[dotstyle=*](0.,-1.)
\psdots[dotstyle=*](-0.5,-1.)
\psdots[dotstyle=*](-1.,-1.)
\psdots[dotstyle=*](1.,-1.)
\psdots[dotstyle=*](0.5,-1.)
\psdots[dotstyle=*](0.,0.)
\psdots[dotstyle=*](1.,-1.)
\psdots[dotstyle=*](0.5,-1.)
\psdots[dotstyle=*](-0.5,-1.)
\end{scriptsize}
\end{pspicture*}
\end{center}

\begin{proof}
\begin{enumerate}[(1)]
\item The subtree $X$ is $\Aut(T)$-invariant. Thus we have a canonical continuous homomorphism $\Aut(T) \to \Aut(X)$. Its kernel is compact and isomorphic to $\prod_{v \in V_8} \Sym(5)$. It contains a characteristic subgroup $K$ isomorphic to $\prod_{v \in V_8} \Alt(5)$. Moreover $\Aut(T)$ has a closed subgroup $S$ isomorphic to $\Aut(X)^+$. The requested subgroup $H$ can be defined as $H = KS$. 

\item Let $N$ be an open normal subgroup of finite index in $H = KS$. Then $N \cap K$ is an open normal subgroup of $K$, and thus contains all but finitely many factors of $\prod_{v \in V_8} \Alt(5)$ (i.e.\ $N \cap K \supseteq \prod_{v \in I} \Alt(5)$ for some cofinite set $I \subseteq V_8$). Since the conjugation action of $H$ is transitive on those factors, we infer that $N$ contains them all. Hence $K \leq N$. Thus the quotient map $H \to H/N$ factors through $H/K \cong S$, which is simple by \cite{Tits_arbre}. Hence $H/N$ is trivial, which confirms that $H = H^{(\infty)}$. 

\item Any infinite topologically simple subgroup of $\Aut(T)$ acts faithfully on $X$. On the other hand, the group $H$ contains an element $h$ fixing a vertex $v \in V_8$ and permuting cyclically the $5$ neighbors of $v$ with degree~$1$. Any closed subgroup $J$ of $\Aut(T)$ which is sufficiently close to $H$ in the Chabauty topology also contains elements fixing $v$ with the same action on its neighbors. In particular the stabilizer $J_v$ has a non-trivial $5$-Sylow subgroup. Since every vertex stabilizer in $\Aut(X)^+$ is a pro-$\{2, 3\}$ group, we deduce that $J$ does not act faithfully on $X$ and is thus not topologically simple. \qedhere
\end{enumerate}
\end{proof}

\subsection{Boundary-\texorpdfstring{$2$}{2}-transitive automorphism groups of trees}\label{section:2-transitive}

Recall that the \textbf{monolith} $\Mon(G)$ of a topological group $G$ is defined to be the (possibly trivial) intersection of all its non-trivial closed normal subgroups. It is clear from Lemma~\ref{lemma:normal} that, when $G$ is infinite, $\Mon(G) \leq G^{(\infty)}$ (because an open subgroup is always closed). If moreover $G$ is totally disconnected and locally compact, then it appears that $\Mon(G) = G^{(\infty)}$ as soon as $\Mon(G)$ is cocompact in $G$.

\begin{lemma}
Let $G$ be a totally disconnected locally compact group. If $G/\Mon(G)$ is compact, then $G^{(\infty)} \leq \Mon(G)$.
\end{lemma}

\begin{proof}
By Lemma~\ref{lemma:normal}, we have
$$G^{(\infty)} = \bigcap_{N \unlhd_{\text{ofi}} G} N.$$
The group $G/\Mon(G)$ is compact by hypothesis and totally disconnected (as a quotient of a totally disconnected locally compact group by a closed subgroup), so it is profinite. In particular, the open (and hence finite index) normal subgroups of $G/\Mon(G)$ form a base of neighborhood of the identity. Their intersection is thus trivial, which implies that the intersection of all open normal subgroups of finite index of $G$ is contained in $\Mon(G)$.
\end{proof}

The previous lemma can be applied when $G$ is a boundary-$2$-transitive automorphism group of a tree, as the following result (due to M. Burger and S. Mozes) shows.

\begin{proposition}\label{proposition:BuMo}
Let $T$ be a locally finite thick semi-regular tree and $H \in \Sub(\Aut(T))$ act $2$-transitively on $\bd T$. Then $H/\Mon(H)$ is compact and $\Mon(H)$ is topologically simple. In particular $\Mon(H) = H^{(\infty)}$. 
\end{proposition}

\begin{proof}
This follows from \cite{BM}*{Propositions~1.2.1 and~3.1.2, Lemma~3.1.1} (see also Proposition~\ref{proposition:monolith} below).
\end{proof}

\begin{corollary}\label{corollary:2transitive}
Let $T$ be a locally finite thick semi-regular tree and $H_n \to H$ be a converging sequence in $\Sub({\Aut(T)})$ whose limit $H$ acts $2$-transitively on $\bd T$. Then we have
$$[H : \Mon(H)] \leq \limsup_{n\to \infty}\, [H_n : H_n^{(\infty)}].$$

In particular, if $H_n$ has no proper open subgroup of finite index for each $n \geq 1$ then $H$ is topologically simple.
\end{corollary}

\begin{proof}
This follows by assembling Corollary~\ref{corollary:tree} and Proposition~\ref{proposition:BuMo}.
\end{proof}

\begin{corollary}\label{corollary:2transitive2}
Let $T$ be a locally finite thick semi-regular tree. The set of topologically simple closed subgroups of $\Aut(T)$ acting $2$-transitively on $\bd T$ is closed in $\Sub(\Aut(T))$.
\end{corollary}

\begin{proof}
It follows easily from \cite{BM}*{Lemma~3.1.1} that the set of boundary-$2$-transitive groups is closed in $\Sub(\Aut(T))$ (and is contained in $\Sub(\Aut(T))_{\leq 2}$). Within that set, the subset of topologically simple groups is closed in view of Corollary~\ref{corollary:2transitive}.
\end{proof}

Recall that, for a locally finite thick tree $T$, we defined the space $\mathcal{S}_T$ by
$$\mathcal{S}_T := \bigslant{\left\{H \in \Sub({\Aut(T)}) \mid \text{$H$ is topologically simple and $2$-transitive on $\bd T$}\right\}}{\cong},$$
where $\cong$ is the relation of \textit{topological isomorphism}. In our context, it actually appears that two groups are topologically isomorphic if and only if they are conjugate in $\Aut(T)$ (see \cite{Radu}*{Proposition~A.1}). This equivalence enables us to show Theorem~\ref{theorem:Bd-2-trans}.

\begin{proof}[Proof of Theorem~\ref{theorem:Bd-2-trans}]
Let $$\mathcal C = \{H \in \Sub({\Aut(T)}) \mid \text{$H$ is topologically simple and $2$-transitive on $\bd T$}\}.$$
By Corollary~\ref{corollary:2transitive2}, the set $\mathcal C$ is closed in $\Sub({\Aut(T)})$ and hence compact Hausdorff.

The set $\bigslant{\mathcal{C}}{\cong}$ endowed with the quotient topology is then also compact. In order to show that this space is Hausdorff, we can simply prove that the quotient map $q \colon \mathcal{C} \to \bigslant{\mathcal{C}}{\cong}$ is open and that the set
$$\mathcal D = \{(H,H') \in \mathcal C \times \mathcal C \mid H \cong H'\}$$
is closed in $\mathcal C \times \mathcal C$.

By \cite{Radu}*{Proposition~A.1}, we have $H \cong H'$ with $H, H' \in \mathcal{C}$ if and only if $H$ and $H'$ are conjugate in $\Aut(T)$. Let $U$ be an open subset of $\mathcal{C}$. We first need to show that $q(U)$ is open, i.e.\ that $q^{-1}(q(U))$ is open. We have
$$q^{-1}(q(U)) = \bigcup_{\sigma \in \Aut(T)} \sigma U \sigma^{-1},$$
and $\sigma U \sigma^{-1}$ is clearly open for each $\sigma \in \Aut(T)$, so $q$ is an open map as wanted. Now consider two sequences $H_n \to H$ and $H'_n \to H'$ in $\mathcal{C}$ with $H_n \cong H'_n$ for all $n \geq 1$, i.e. $H'_n = \sigma_n H_n \sigma_n^{-1}$ for some $\sigma_n \in \Aut(T)$. As $H_n$ is edge-transitive, we can assume that $\sigma_n$ sends a fixed vertex $v_0$ to a vertex at distance $\leq 1$ for all $n \geq 1$. Hence, $(\sigma_n)$ subconverges to some $\sigma \in \Aut(T)$ and $H' = \sigma H \sigma^{-1}$ by Lemma~\ref{lemma:conjugate}. So $\mathcal D$ is closed in $\mathcal C \times \mathcal C$.

Finally, the fact that $\mathcal D$ is closed in $\mathcal C \times \mathcal C$ also implies that $\cong$ has closed classes.
\end{proof}
\subsection{Local prime content and local torsion-freeness}

Let $T$ be a locally finite tree. In this section, we provide applications of Corollary~\ref{corollary:tree:convergence<k-closure} by highlighting two algebraic properties that define open subsets of the   space 
$$\Sub(\Aut(T))_{\leq C}^{0} := \{H \in \Sub(\Aut(T))_{\leq C} \mid H \text{ is unimodular}\}.$$

Let $\pi$ be a set of primes. A totally disconnected locally compact group is called \textbf{locally pro-$\pi$} if it has an open pro-$\pi$ subgroup. 
If $G$ is the full automorphism group of a regular rooted tree, then the set of locally pro-$\pi$ subgroups is generally neither open nor closed in the Chabauty space $\Sub(G)$. The following result shows that this situation changes if one considers closed subgroups of bounded covolume in $\Aut(T)$.

\begin{proposition}\label{proposition:LocalPrimeContent}
Let $T$ be a locally finite tree all of whose vertices have degree~$\geq 2$ and let $C >0$. Then for any set of primes $\pi$, the set of locally pro-$\pi$ groups is open in $\Sub(\Aut(T))_{\leq C}^{0}$. 
In particular the set of discrete subgroups is open in $\Sub(\Aut(T))_{\leq C}^{0}$. 
\end{proposition}

\begin{proof}
Let $H$ be a locally pro-$\pi$ group in $\Sub(\Aut(T))_{\leq C}^{0}$. By Proposition~\ref{proposition:locallyProPi}, there exists $K \geq 0$ such that $\closure{K}{H}$ is also locally pro-$\pi$. We also know from Corollary~\ref{corollary:tree:convergence<k-closure} that the set
$$\{J \in \Sub(\Aut(T)) \mid \sigma J \sigma^{-1} \leq \closure{K}{H} \text{ for some } \sigma \in \Aut(T) \}$$
is a neighborhood of $H$ in $\Sub(\Aut(T))$. This set only contains locally pro-$\pi$ groups, so the conclusion follows.
\end{proof}

\begin{remark}
We emphasize that the set of locally pro-$\pi$ groups in $\Sub(\Aut(T))_{\leq C}^{0}$ need not be closed in general. In order to see that, let $T$ be the $d$-regular tree, with $d \geq 3$. By \cite{Bass-Kulkarni}*{Theorem 7.1 (a)}, the group $\Aut(T)$ contains a properly ascending chain of cocompact lattices $\Gamma_1 < \Gamma_2 < \dots$. Denoting by $C$ the number of $\Gamma_1$-orbits of vertices, we have $\Gamma_i \in \Sub(\Aut(T))_{\leq C}^{0}$ for all $i$. Let $H = \overline{\bigcup_{i \geq 1} \Gamma_i}$. Since $\Gamma_1$ is a lattice in $H$, it follows that $H$ is unimodular, so that $H \in \Sub(\Aut(T))_{\leq C}^{0}$. If $H$ were discrete, it would be a cocompact lattice in $\Aut(T)$, and the chain of inclusions $\Gamma_1 \leq H \leq \Aut(T)$ would force the index $[H : \Gamma_1]$ to be finite, contradicting the properly ascending property of the chain $\Gamma_1 < \Gamma_2 < \dots$. We infer that $H$ is non-discrete. In particular $H$ is not locally pro-$\varnothing$. On the other hand, we have $H = \lim_{i \to \infty} \Gamma_i$ by Lemma~\ref{lemma:inter} (2). Hence we have constructed a converging sequence of locally pro-$\varnothing$ groups in $\Sub(\Aut(T))_{\leq C}^{0}$, whose limit is not locally pro-$\varnothing$. This confirms that the set of locally pro-$\pi$ groups in $\Sub(\Aut(T))_{\leq C}^{0}$ is not closed in general. 
\end{remark}

A totally disconnected locally compact group is called \textbf{locally torsion-free} if it has an open torsion-free subgroup. Typical examples are provided by $p$-adic analytic groups (see \cite{AnalyticPro}*{Theorems~4.5 and~8.1}). 

We shall need the following basic fact. 

\begin{lemma}\label{lemma:TorsionFreeProfinite}
Let $U$ be a torsion-free profinite group and $U = U_0 \geq U_1 \geq \dots $ be a descending chain of open subgroups of $U$ with trivial intersection. Let also $p$ be a positive integer.
For each $m \geq 0$, there exists $M$ such that for all $u \in U$, if $u^p \in U_M$ then $u \in U_m$. 
\end{lemma}

\begin{proof}
Suppose the contrary. Then there exist $m \geq 0$, a sequence of integers $(k_n)$ tending to infinity with $n$, and a sequence $(u_n)$ in $U$ such that $u_n^p \in U_{k_n}$ and $u_n \not \in U_m$. Upon extracting, we may assume without loss of generality that $(u_n)$ converges to some $u \in U$. Since $U_m$ is open and $u_n \not \in U_m$ for all $n$, we also have $u \not \in U_m$. In particular $u \neq 1$. On the other hand, we have $u_n^p \in U_{k_n}$, so that $u^p = (\lim_n u_n)^p = \lim_n u_n^p = 1$. Hence $u$ is a non-trivial torsion element of $U$, a contradiction. 
\end{proof}

\begin{proposition}\label{proposition:LocallyTorsionFree}
Let $T$ be a locally finite tree all of whose vertices have degree~$\geq 2$ and let $C > 0$. Then the set of locally torsion-free groups is open in $\Sub(\Aut(T))_{\leq C}^{0}$.
\end{proposition}

\begin{proof}
We can suppose that $\Sub(\Aut(T))_{\leq C}^{0}$ is non-empty. It follows that $T$ is of bounded degree. We define the finite set of primes  
$\pi = \{ p \text{ prime } |\ p \leq \deg(v)\  \forall v \in V(T)\}$ 
and observe that the stabilizer $\Aut(T)_v$ of any vertex $v \in V(T)$ is a pro-$\pi$ group.

Let $H \in \Sub(\Aut(T))_{\leq C}^{0}$ be locally torsion-free. We must show that $H$ has a neighborhood in $\Sub(\Aut(T))_{\leq C}^{0}$ that consists of locally torsion-free groups. 

\setcounter{claim}{0}
\begin{claim}\label{cl1}
There exist integers $M > n > 0$ such that for all $v \in V(T)$, $p \in \pi$ and $h \in H_v^{[n]}$, if $h^p \in H_v^{[M]}$ then $h \in H_v^{[n+1]}$. 
\end{claim}

\begin{claimproof}
Since $H$ is locally torsion-free, there exists $n_0 \geq 0$ and $v_0 \in V(T)$ such that $H_{v_0}^{[n_0]}$ is torsion-free. As $H$ acts cocompactly on $T$, there exists $n \geq n_0$ such that $H_v^{[n]}$ is torsion-free for all $v \in V(T)$.

Let us now fix a prime $p \in \pi$ and a vertex $v \in V(T)$ and let us apply Lemma~\ref{lemma:TorsionFreeProfinite} to the torsion-free profinite groups $U_k = H_v^{[n+k]}$ and the integer $m = 1$. This yields a constant $M(p, v) \geq n$ such that for all $h \in H_v^{[n]}$, if $h^p \in H_v^{[M(p,v)]}$ then $h \in H_v^{[n+1]}$. 

We next define $M$ as the supremum of $M(p, v)$ taken over all $p \in \pi$ and all vertices $v$ in a (necessarily finite) fundamental domain for the $H$-action on $V(T)$. Then the required property holds (and we can assume that $M > n$). 
\end{claimproof}

\begin{claim}\label{cl2}
Let $M > n$ be the constants afforded by Claim~\ref{cl1} and let $G = \closure{M}{H}$. Then the group $G_v^{[n]}$ is torsion-free for all $v \in V(T)$. In particular $G$ is locally torsion-free. 
\end{claim}

\begin{claimproof}
Suppose for a contradiction that for some $v \in V(T)$, there exists a non-trivial torsion element in $G_v^{[n]}$. By the definition of $\pi$, every non-trivial torsion element of $\Aut(T)$ has a power which is a non-trivial element of order $p$ for some $p \in \pi$. We may thus assume that $G_v^{[n]}$ contains a non-trivial element $g$ of prime order $p \in \pi$. Let then $k \geq n$ be the largest integer such that $g \in G_v^{[k]}$. Then there exists a vertex $x \in B(v, k)$ fixed by $g$, such that $g$ does not fix $B(x, 1)$ pointwise. It follows that there exists a vertex $y$ on the geodesic segment joining $v$ to $x$ such that $g \in G_y^{[n]}$ and $g \not \in G_y^{[n+1]}$. 

Since $G = \closure{M}{H}$, there exists $h \in H$ such that $g|_{B(y, M)} = h|_{B(y, M)}$. The properties that $g \in G_y^{[n]}$, that $g^p = 1$ and that $g \not \in G_y^{[n+1]}$ respectively imply that $h \in H_y^{[n]}$, that $h^p \in H_y^{[M]}$ and that $h \not \in H_y^{[n+1]}$. This contradicts Claim~\ref{cl1}.
\end{claimproof}

\medskip

From Claim~\ref{cl2} we know that $\closure{M}{H}$ is locally torsion-free. Hence, the set
$$\{J \in \Sub(\Aut(T)) \mid \sigma J \sigma^{-1} \leq \closure{M}{H} \text{ for some } \sigma \in \Aut(T) \},$$
which is a neighborhood of $H$ in $\Sub(\Aut(T))$ by Corollary~\ref{corollary:tree:convergence<k-closure},
only contains locally torsion-free groups
\end{proof}

\begin{remark}
We emphasize that the set of locally torsion-free groups in $\Sub(\Aut(T))_{\leq C}^{0}$ need not be closed in general. An excellent illustration of that fact is provided by the main results of \cite{Stulemeijer}, showing that some simple algebraic groups over local fields of positive characteristic (which are not locally torsion-free) are Chabauty limits of simple algebraic groups over $p$-adic fields (which are $p$-adic analytic, hence locally torsion-free). 
\end{remark}

\section{Buildings}

\subsection{Weyl-transitive automorphism groups of buildings}

Let $\Delta$ be a locally finite thick building. 
A subgroup $H$ of $\Aut(\Delta)$ is said to be \textbf{Weyl-transitive} if, for all $w \in W$, the action of $H$ on the ordered pairs $(c_1,c_2)$ of chambers such that $\delta(c_1,c_2)=w$ is transitive, where $\delta \colon \Ch(\Delta)\times \Ch(\Delta) \to W$ is the Weyl-distance. 

\begin{remark}
If $H \leq \Aut(\Delta)$ is \textbf{strongly transitive} on $\Delta$ (i.e.\ transitive on pairs $(A, c)$ consisting of an apartment $A$ and a chamber $c \in A$), then it is Weyl-transitive. The converse holds if $\Delta$ is spherical, but not in general: see \cite{Abramenko}*{Proposition~6.14}. If $\Delta$ is of affine type (e.g. $\Delta$ is a tree) and $H$ is closed, it may be seen that if $H$ is Weyl-transitive, then it is strongly transitive on the spherical building at infinity of $\Delta$, hence strongly transitive on $\Delta$ by \cite{CaCi}*{Theorem~1.1}. For $\Delta$ arbitrary (e.g. hyperbolic), the existence of Weyl-transitive but non-strongly transitive closed subgroups 
$H \leq \Aut(\Delta)$ is likely, but currently we do not know explicit examples. 
\end{remark}

The following result, which is a straightforward adaptation of \cite{Caprace-Monod}*{Corollary~3.1} dealing with strongly transitive actions, shows that monolithic groups naturally appear in the context of Weyl-transitive automorphism groups of buildings. It may be seen as a generalization of Proposition~\ref{proposition:BuMo}.

\begin{proposition}\label{proposition:monolith}
Let $\Delta$ be an infinite irreducible locally finite thick building and $H \in \Sub(\Aut(\Delta))$ be Weyl-transitive. Then $\Mon(H)$ is topologically simple and transitive on the set of chambers of $\Delta$. In particular, $H/\Mon(H)$ is compact and $\Mon(H) = H^{(\infty)}$. 
\end{proposition}

\begin{proof}
We follow the proof of \cite{Caprace-Monod}*{Corollary~3.1}. In generalizing from strongly transitive to Weyl-transitive actions, the point requiring a supplementary check is that Tits' transitivity lemma, which was originally stated for strongly transitive actions, holds more generally for Weyl-transitive action. This is indeed the case by \cite{Abramenko}*{Lemma~6.61}. We are thus ensured that any non-trivial normal subgroup of $H$ is transitive on the set of chambers of $\Delta$. Therefore any non-trivial closed normal subgroup of $H$ is cocompact. Since $H$ is Weyl-transitive on $\Delta$, it is chamber-transitive, hence compactly generated. We may then invoke \cite{Caprace-Monod}*{Theorem~E}, and conclude the proof word-by-word as in \cite{Caprace-Monod}*{Corollary~3.1}. The argument can be summarized as follows. We know from \cite{Caprace-Monod}*{Theorem~E} that the monolith of $H$ is a quasi-product of topologically simple groups. However, there can be only one simple factor using that the building $\Delta$ has locally compact CAT(0) metric realization. The desired assertions follow.
\end{proof}

The next corollary is then a direct consequence of Theorem~\ref{theorem:simpleclosed}.

\begin{corollary}\label{corollary:Building}
Let $\Delta$ be an infinite irreducible locally finite thick building and $\Gamma \leq \Aut(\Delta)$ act cocompactly on $\Delta$. Let $H_n \to H$ be a converging sequence in $\Sub({\Aut(\Delta)})$ whose limit $H$ is Weyl-transitive. Suppose that for each $n \geq 1$, there exists $\tau_n \in \Aut(\Delta)$ such that $\tau_n \Gamma \tau_n^{-1} \leq H_n$. Then we have
$$[H : \Mon(H)] \leq \limsup_{n\to \infty}\, [H_n : H_n^{(\infty)}].$$
In particular, if $H_n$ has no proper open subgroup of finite index for each $n \geq 1$ then $H$ is topologically simple.
\end{corollary}

\begin{proof}
This follows from Proposition~\ref{proposition:monolith} and Theorem~\ref{theorem:simpleclosed}, since a locally finite building can be seen as a locally finite connected graph whose vertices are the chambers and whose edges are the pairs of adjacent chambers.
\end{proof}

\begin{remark}
If $\Delta$ is a tree, then a closed Weyl-transitive subgroup of $\Aut(\Delta)$ is $2$-transitive on the set of ends $\bd \Delta$. Thus Corollary~\ref{corollary:2transitive} can be deduced from Corollary~\ref{corollary:Building}.
\end{remark}

\begin{remark}
If $\Delta$ is a locally finite Euclidean building of dimension~$\geq 2$, it can be seen that there is a unique topologically simple closed subgroup of $\Aut(\Delta)$ acting Weyl-transitively: namely the simple algebraic group to which $\Delta$ is associated via Bruhat--Tits theory. This is of course not the case for trees. For higher-dimensional more exotic buildings (e.g. Bourdon buildings), there can be a much larger collection of simple groups acting Weyl-transitively, whose variety might potentially be comparable to one encountered in the case of trees (see \cite{DMSiSt}). 
\end{remark}

\subsection{Buildings of virtually free type}

We have seen in \S\ref{subsection:trees} that, for trees, the condition about the common cocompact group $\Gamma$ was always fulfilled. It appears that, more generally, it is possible to drop the hypothesis about $\Gamma$ in the context of buildings whose associated Coxeter group is virtually free. The reason is the existence of a strong relation between such buildings and trees.

\begin{lemma}\label{lemma:virtuallyfree}
Let $\Delta$ be an infinite irreducible locally finite thick building whose Weyl group $W$ is virtually free. Suppose that $\Aut(\Delta)$ is chamber-transitive. Then there exists a locally finite tree $T$ on which $\Aut(\Delta)$ acts continuously, properly, faithfully and cocompactly.
\end{lemma}

\begin{proof}
By \cite{Davis}*{Proposition~8.8.5}, $W$ is virtually free if and only if $W$ has a tree of groups decomposition where each vertex group is a spherical special subgroup. If $\mathcal{X}$ is the tree of groups, then we write $X$ for the underlying tree and denote $W = \pi_1(\mathcal{X})$.

Since $\Aut(\Delta)$ is chamber-transitive, we have by~\cite{Tits}*{Proposition~2} that $\Aut(\Delta) = \pi_1(\mathcal{X}_0)$ where $\mathcal{X}_0$ has the same underlying tree $X$ as $\mathcal{X}$ and has adequate residue stabilizers as vertex groups and edge groups. By~\cite{Serre}*{\S I.4.5, Th\'eor\`eme 9}, we deduce that $\Aut(\Delta)$ acts on a locally finite tree $T$ in such a way  that $\Aut(\Delta) \backslash T = X$. Moreover, the stabilizer of a vertex of $T$ in $\Aut(\Delta)$ corresponds to a stabilizer of a spherical residue of $\Delta$ and hence is compact and open. This implies that the action of $\Aut(\Delta)$ on $T$ is continuous and proper. Finally, the kernel $K \leq \Aut(\Delta)$ of this action on $T$ stabilizes all residues of $\Delta$ of a fixed spherical type. Since $\Delta$ is infinite and irreducible, this implies that $K$ is trivial (see~\cite{AB-unbounded}*{Main Theorem}). The action is thus faithful.
\end{proof}

\begin{remark}
The tree of group decomposition of $W$ is generally not unique. In particular, the tree $T$ and the $\Aut(\Delta)$-action on $T$ afforded by Lemma~\ref{lemma:virtuallyfree} are not canonical. 
\end{remark}

\begin{corollary}\label{corollary:virtuallyfree}
Let $\Delta$ be an infinite irreducible locally finite thick building of virtually free type $W$. Let $H_n \to H$ be a converging sequence in $\Sub({\Aut(\Delta)})$ whose limit $H$ is Weyl-transitive. Then we have
$$[H : \Mon(H)] \leq \limsup_{n\to \infty}\, [H_n : H_n^{(\infty)}].$$

In particular, if $H_n$ has no proper open subgroup of finite index for each $n \geq 1$ then $H$ is topologically simple.
\end{corollary}

\begin{proof}
Let $T$ be the locally finite tree given by Lemma~\ref{lemma:virtuallyfree}. The fact that $\Aut(\Delta)$ acts continuously, properly and faithfully on $T$ means that there is a map $i \colon \Aut(\Delta) \to \Aut(T)$ which is an isomorphism onto its image, the latter being closed in $\Aut(T)$. We thus have the converging sequence $i(H_n) \to i(H)$ in $\Sub({\Aut(T)})$, such that $i(H)$ acts cocompactly on $T$ and is unimodular (because it is generated by compact subgroups). The conclusion then follows from Corollary~\ref{corollary:tree}, since $\Mon(H) = H^{(\infty)}$ (see Proposition~\ref{proposition:monolith}).
\end{proof}

\appendix

\section{The clopen subset \texorpdfstring{$\mathcal{S}_T^{\Alt} \subseteq \mathcal{S}_T$}{S\_T}}\label{sec:appendix}

Let $T$ be a locally finite thick semi-regular tree. As before, we denote by $\mathcal{S}_T$ the set of isomorphism classes of groups in $\Sub(\Aut(T))$ which are topologically simple and $2$-transitive on $\bd T$. By Theorem~\ref{theorem:Bd-2-trans}, this set carries a compact Hausdorff topology induced from the Chabauty topology on $\Sub(\Aut(T))$. The goal of this appendix is to provide supplementary information on that compact space.

First recall that, when $H \in \Sub({\Aut(T)})$ is $2$-transitive on $\bd T$, the action of the stabilizer $H_v$ of a vertex $v \in V(T)$ in $H$ is $2$-transitive on the set of neighbors of $v$ (see \cite{BM}*{Lemma~3.1.1}). In particular, $H$ must be edge-transitive.

Recall also that $\mathcal{S}_T$ contains the isomorphism class of $\Aut(T)^+$ (which is simple by \cite{Tits_arbre}). In~\cite{Radu}, the second-named author restricted his attention to the groups $H$ which locally contain the full alternating group, i.e.\ such that the action of $H_v$ on its set of $d$ neighbors contains $\Alt(d)$ for each $v \in V(T)$. Let us denote by $\mathcal{S}_T^{\Alt}$ the subset of $\mathcal{S}_T$ consisting of the isomorphism classes of all those groups. An exhaustive description of the set $\mathcal{S}_T^{\Alt}$ when the vertices of $T$ have degree $\geq 6$ is given in \cite{Radu}. Below we summarize some of its properties.

\begin{proposition}\label{proposition:Cantor-Bendixson}
Let $T$ be the $(d_0, d_1)$-semi-regular tree. Then $\mathcal{S}_T^{\Alt}$ is a closed open subset of $\mathcal{S}_T$ containing the isomorphism class of $\Aut(T)^+$. 

Moreover, if $d_0, d_1 \geq 6$, then the compact space $\mathcal{S}_{T}^{\Alt}$ is countably infinite and its second Cantor--Bendixson derivative is $\{[\Aut(T)^+]\}$.
\end{proposition}

\begin{proof}
The first assertion is clear. In order to prove the second one, we freely use the terminology and notation from \cite{Radu} without repeating all the definitions in full details. In that paper, a \textit{legal coloring} $i$ of $T$ is fixed and, given two possibly empty finite subsets $Y_0, Y_1 \subset \N$, a group $G_{(i)}^+(Y_0,Y_1)$ is defined. Let us describe their properties which will be needed here. We assume henceforth that $d_0, d_1 \geq 6$. The group $G_{(i)}^+(\varnothing, \varnothing)$ is exactly $\Aut(T)^+$, while $G_{(i)}^+(\{0\}, \{0\})$ is the semiregular analog of the universal locally alternating group of Burger--Mozes \cite{BM}. We call it $U_{(i)}^+(\Alt)$. For all $Y_0$ and $Y_1$ we have $U_{(i)}^+(\Alt) \leq G_{(i)}^+(Y_0, Y_1) \leq \Aut(T)^+$. The groups $G_{(i)}^+(Y_0, Y_1)$ locally contain the alternating group (since they contain $U_{(i)}^+(\Alt)$ which does), are boundary-$2$-transitive and abstractly simple (see \cite{Radu}*{Theorem~A (i), (ii)}) and, up to conjugation, these are the only such groups (see~\cite{Radu}*{Theorem~B (ii)}). If $[G]$ denotes the isomorphism class of $G \in \Sub(\Aut(T))$, then this means that
$$\mathcal{S}_T^{\Alt} = \{[G_{(i)}^+(Y_0, Y_1)] \mid Y_0, Y_1 \subset \N \text{ are finite}\}.$$
In fact, the groups $G_{(i)}^+(Y_0, Y_1)$ are not pairwise distinct, but this is not important for the following discussion. We now give some other properties of these groups. In order to shorten the statements, we adopt the convention $\max(\varnothing) := +\infty$.

\begin{fact}\label{fact:convergence}
If $(X^{(n)})$ is a sequence of finite subsets of $\N$ such that $\max X^{(n)} \to +\infty$ and if $Y$ is a finite subset of $\N$, then $G_{(i)}^+(X^{(n)}, Y) \to G_{(i)}^+(\varnothing, Y)$ and $G_{(i)}^+(Y, X^{(n)}) \to G_{(i)}^+(Y, \varnothing)$.
\end{fact}

\begin{factproof}
When $X$ is a non-empty finite subset of $\N$ and $Y$ is a finite subset of $\N$, we have $G_{(i)}^+(X, Y) \leq G_{(i)}^+(\varnothing, Y) \leq \closure{(\max X)}{G_{(i)}^+(X, Y)}$ (see \cite{Radu}*{\S4.1}). If $(X^{(n)})$ is a sequence such that $\max X^{(n)} \to +\infty$, then we can deduce that $G_{(i)}^+(X^{(n)}, Y) \to G_{(i)}^+(\varnothing, Y)$ with Lemma~\ref{lemma:converge}. Indeed, (i) is clear and (ii) can be obtained as follows. Fix $h \in G_{(i)}^+(\varnothing, Y)$ and $v_0 \in V(T)$. For each $n \geq 1$, since $G_{(i)}^+(\varnothing, Y) \leq \closure{(\max X^{(n)})}{G_{(i)}^+(X^{(n)}, Y)}$ (we can assume that $X^{(n)}$ is non-empty), there exists $h_n \in G_{(i)}^+(X^{(n)}, Y)$ such that $h|_{B(v_0, \max X^{(n)})} = h_n|_{B(v_0, \max X^{(n)})}$. Then $h_n \to h$ because $\max X^{(n)} \to +\infty$, which proves (ii). The reasoning is exactly the same to obtain that $G_{(i)}^+(Y, X^{(n)}) \to G_{(i)}^+(Y, \varnothing)$.
\end{factproof}

\begin{fact}\label{fact:inclusion}
If $\sigma G_{(i)}^+(Z_0, Z_1) \sigma^{-1} \leq G_{(i)}^+(Y_0, Y_1)$ for some finite subsets $Y_0, Y_1, Z_0, Z_1 \subset \N$ and some $\sigma \in \Aut(T)$, then either $G_{(i)}^+(Z_0, Z_1) \leq G_{(i)}^+(Y_0, Y_1)$ or $G_{(i)}^+(Z_1, Z_0) \leq G_{(i)}^+(Y_0, Y_1)$.
\end{fact}

\begin{factproof}
If $\sigma \in \Aut(T)^+$ then $G_{(i)}^+(Z_0, Z_1) \leq G_{(i)}^+(Y_0, Y_1)$ by \cite{Radu}*{Lemma~4.10 (i)}. If $\sigma \in \Aut(T) \setminus \Aut(T)^+$, then there exists a particular element $\nu \in \Aut(T) \setminus \Aut(T)^+$ such that $\nu G_{(i)}^+(Z_0, Z_1) \nu^{-1} = G_{(i)}^+(Z_1, Z_0)$ and the conclusion follows.
\end{factproof}

\begin{fact}\label{fact:inclusion2}
If $G_{(i)}^+(Z_0, Z_1) \leq G_{(i)}^+(Y_0, Y_1)$ for some finite subsets $Y_0, Y_1, Z_0, Z_1 \subset \N$, then $\max Z_0 \leq \max Y_0$ and $\max Z_1 \leq \max Y_1$.
\end{fact}

\begin{factproof}
In \cite{Radu}*{\S 5.3}, two invariants $K'_H(0), K'_H(1) \in \N \cup \{+\infty\}$ are associated to any closed subgroup $H \leq \Aut(T)$ containing $U_{(i)}^+(\Alt)$. These invariants have the property that if $H \leq H'$ then $K'_H(0) \leq K'_{H'}(0)$ and $K'_H(1) \leq K'_{H'}(1)$. For $H = G_{(i)}^+(Y_0,Y_1)$, we have $K'_H(0) = \max Y_0$ and $K'_H(1) = \max Y_1$ (see \cite{Radu}*{Table~1}), which suffices to conclude.
\end{factproof}

\begin{fact}\label{fact:closed}
For all finite subsets $Y_0, Y_1 \subset \N$, there exists an integer $K \geq 0$ such that $\closure{K}{G_{(i)}^+(Y_0,Y_1)} = G_{(i)}^+(Y_0,Y_1)$.
\end{fact}

\begin{factproof}
See~\cite{Radu}*{Theorem~H}.
\end{factproof}

\medskip

Let us now compute the Cantor-Bendixson derivatives of $\mathcal{S}_{T}^{\Alt}$. In $\mathcal{S}_{T}^{\Alt}$, the points $[G_{(i)}^+(Y, \varnothing)]$ and $[G_{(i)}^+(\varnothing, Y)]$ are not isolated. Indeed, $[G_{(i)}^+(Y, \{n\})] \to [G_{(i)}^+(Y, \varnothing)]$ when $n \to +\infty$ by Fact~\ref{fact:convergence}, and $[G_{(i)}^+(Y, \{n\})] \neq [G_{(i)}^+(Y, \varnothing)]$ for each $n \geq 0$ (Facts~\ref{fact:inclusion} and~\ref{fact:inclusion2}). We claim that the points $[G_{(i)}^+(Y_0, Y_1)]$ with $Y_0$ and $Y_1$ non-empty are isolated. Suppose for a contradiction that there exists sequences $(Y_0^{(n)})$, $(Y_1^{(n)})$ and $(\tau_n)$ with $\tau_n \in \Aut(T)$ such that $\tau_n G_{(i)}^+(Y_0^{(n)}, Y_1^{(n)}) \tau_n^{-1} \to G_{(i)}^+(Y_0, Y_1)$ and with $[G_{(i)}^+(Y_0^{(n)}, Y_1^{(n)})] \neq [G_{(i)}^+(Y_0, Y_1)]$ for each $n \geq 1$. Then, by Fact~\ref{fact:closed} and Corollary~\ref{corollary:tree:convergence<k-closure}, there exists $(\sigma_n)$ with $\sigma_n \in \Aut(T)$ such that $\sigma_n G_{(i)}^+(Y_0^{(n)}, Y_1^{(n)}) \sigma_n^{-1} \leq G_{(i)}^+(Y_0, Y_1)$ for sufficiently large $n$. From Fact~\ref{fact:inclusion}, we deduce that $G_{(i)}^+(Y_0^{(n)}, Y_1^{(n)}) \leq G_{(i)}^+(Y_0, Y_1)$ or $G_{(i)}^+(Y_1^{(n)}, Y_0^{(n)}) \leq G_{(i)}^+(Y_0, Y_1)$ for all sufficiently large $n$. But there are only finitely many $Z_0, Z_1 \subset \N$ such that $G_{(i)}^+(Z_0, Z_1) \leq G_{(i)}^+(Y_0, Y_1)$ (see Fact~\ref{fact:inclusion2}), so we cannot have the supposed convergence. We just proved that the first Cantor-Bendixson derivative of $\mathcal{S}_{T}^{\Alt}$ is
$$(\mathcal{S}_T^{\Alt})' = \{[G_{(i)}^+(Y, \varnothing)], [G_{(i)}^+(\varnothing, Y)] \mid Y \subset \N \text{ is finite}\}.$$
With the exact same reasoning, we then obtain that
\[(\mathcal{S}_T^{\Alt})'' = \{[G_{(i)}^+(\varnothing, \varnothing)]\} = \{[\Aut(T)^+]\}. \qedhere \]
\end{proof}

\begin{remark}
Proposition~\ref{proposition:Cantor-Bendixson} implies (see~\cite{MS20}*{Th\'eor\`eme 1}) that $\mathcal{S}_{T}^{\Alt}$ is homeomorphic to the space $\hat{\Z}^2$, where $\hat{\Z}$ is the one-point compactification of the discrete space $\Z$ (which is homeomorphic to the compact subset $\left \{1, \frac 1 2, \frac 1 3, \dots, 0\right \}$ of the real line). 
\end{remark}

\begin{remark}\label{remark:S^Alt}
If the degrees $d_0$ and $d_1$ of the vertices of $T$ are such that each $2$-transitive subgroup of $\Sym(d_t)$ contains $\Alt(d_t)$ (for each $t \in \{0,1\}$), then $\mathcal{S}_T = \mathcal{S}_T^{\Alt}$. For example, this is the case for $d_0 = d_1 = 3$. If in addition $d_0, d_1 \geq 6$, then the classification theorem from \cite{Radu} applies, and therefore yields a complete description of $\mathcal{S}_T$. One should note that the set of natural numbers $d$ such that each $2$-transitive subgroup of $\Sym(d)$ contains $\Alt(d)$ is asymptotically dense in $\N$ (see~\cite{Radu}*{Corollary~B.2}).
\end{remark}

\begin{remark}
It is actually a direct consequence of~\cite{Radu}*{Theorem~A (i), (ii) and Theorem~B (i)} that the space $\mathcal{S}_T^{\Alt}$ (and hence also $\mathcal{S}_T$) is infinite when $d_0, d_1 \geq 4$. The case where $d_0 = 3$ or $d_1 = 3$ is not explicitly dealt with in \cite{Radu}, but one can show that $\mathcal{S}_T^{\Alt}$ is infinite also in that case. Indeed, the definition of the groups $G_{(i)}^+(Y_0,Y_1)$ (where $Y_0, Y_1$ are finite subsets of $\N$) from loc.\ cit.\ makes sense for all $d_0,d_1 \geq 3$. For these groups to be boundary-$2$-transitive, one however needs to require $Y_0 \neq \{0\}$ (resp. $Y_1 \neq \{0\}$) when $d_0 = 3$ (resp. $d_1 = 3$). Under the latter hypothesis, it is then possible to adapt the ideas from \cite{Radu}*{\S4} and show that these groups are abstractly simple and that they represent infinitely many isomorphism classes. In the specific case of the trivalent tree $T_3$, the infiniteness of $\mathcal S_{T_3}$ can alternatively be established using rank one simple algebraic groups over local fields with residue field of order~$2$. An exhaustive description of the subset of $\mathcal S_{T_3}$ consisting of (isomorphism classes of) algebraic groups may be found in \cite{Stulemeijer}.
\end{remark}


\end{document}